\documentclass[10pt,a4paper, twoside]{article}
\usepackage[utf8]{inputenc}
\usepackage{hyperref}
\usepackage{amsmath, amsthm}
\usepackage{amsfonts}
\usepackage{amssymb}
\usepackage{fixltx2e}
\usepackage{mathtools}
\usepackage[hyperref, dvipsnames]{xcolor}
\hypersetup{
  colorlinks=true,
  citecolor=Cerulean,
  linkcolor=purple!80
}

\usepackage{enumitem}
\usepackage[nameinlink,noabbrev,capitalize]{cleveref} 

\usepackage[left=2cm,right=2cm,top=2cm,bottom=2cm]{geometry}
\usepackage{comment}
\usepackage[normalem]{ulem}

\usepackage{biblatex}
\newcommand{\grad}{\nabla}

\newcommand{\norm}[2]{\mathopen{}\left\lVert #1\right\rVert_{#2}}
\DeclarePairedDelimiterX\dual[2]{\langle}{\rangle}{#1,#2}
\DeclarePairedDelimiter\parens()
\DeclarePairedDelimiter\abs{\lvert}{\rvert}

\providecommand\given{\nonscript\;\delimsize|\nonscript\;\mathopen{}}
\DeclarePairedDelimiterX\set[1]\{\}{#1}

\newcommand{\weaklyto}{\rightharpoonup}

\newcommand{\Lip}[1]{\mathrm{Lip}(#1)}
\newcommand{\cts}{\hookrightarrow}
\newcommand{\ctsCompact}{\xhookrightarrow{c}}
\newcommand{\ctsDense}{\xhookrightarrow{d}}

\newcommand{\Uad}{U_{\textup{ad}}}

\newcommand{\m}{\mathsf{m}}
\newcommand{\M}{\mathsf{M}}
\newcommand{\Id}{\mathrm{Id}}
\newcommand{\Z}{\mathsf{Z}}

\newcommand\N{\mathbb{N}}
\newcommand\R{\mathbb{R}}

\definecolor{darkorange}{rgb}{8.,.4,0.}

\newtheorem{theorem}{Theorem}[section]
\newtheorem{prop}[theorem]{Proposition}
\newtheorem{lem}[theorem]{Lemma}
\newtheorem{cor}[theorem]{Corollary}
\newtheorem{eg}[theorem]{Example}
\newtheorem{remark}[theorem]{Remark}
\newtheorem{ass}[theorem]{Assumption}

\theoremstyle{definition}
\newtheorem{defn}[theorem]{Definition}

\title{Minimal and maximal solution maps of elliptic QVIs of obstacle type: Lipschitz stability, differentiability and optimal control}
\author{Amal Alphonse\thanks{Weierstrass Institute, Mohrenstrasse 39, 10117 Berlin, Germany (\href{mailto:alphonse@wias-berlin.de}{alphonse@wias-berlin.de})} 
\and Michael Hinterm\"{u}ller\thanks{Weierstrass Institute, Mohrenstrasse 39, 10117 Berlin, Germany (\href{mailto:hintermueller@wias-berlin.de}{hintermueller@wias-berlin.de})} 
\and Carlos N. Rautenberg\thanks{Department of Mathematical Sciences and the Center for Mathematics and Artificial Intelligence (CMAI), George Mason University, Fairfax, VA 22030, USA (\href{mailto:crautenb@gmu.edu}{crautenb@gmu.edu})} 
\and Gerd Wachsmuth\thanks{ Institute of Mathematics, Brandenburgische Technische Universität Cottbus-Senftenberg, 03046 Cottbus, Germany (\href{mailto:gerd.wachsmuth@b-tu.de}{gerd.wachsmuth@b-tu.de})}}

\begin{document}
\maketitle
\begin{abstract}
Quasi-variational inequalities (QVIs) of obstacle type in many cases have
multiple solutions that can be ordered.
We study a multitude
of properties of the operator mapping the source term to the minimal or maximal
solution of such QVIs.
We prove that the solution
maps are locally Lipschitz continuous and directionally differentiable and show
existence of optimal controls for problems that incorporate these maps as the
control-to-state operator. We also consider a Moreau--Yosida-type penalisation for the
QVI wherein we show that it is possible to approximate the minimal and maximal
solutions by sequences of minimal and maximal solutions (respectively) of
certain PDEs, which have a simpler structure and offer a convenient
characterisation in particular for computation. For solution mappings of these
penalised problems, we prove a number of properties including   Lipschitz and
differential stability. Making use of the penalised equations, we derive (in
the limit) C-stationarity conditions for the control problem, in addition to
the Bouligand stationarity we get from the differentiability result.
\end{abstract}
\tableofcontents

\section{Introduction}\label{sec:intro}
Let $(\Omega, \sigma, \vartheta)$ be a measure space and define $H:=L^2(\Omega)$ to be the usual Lebesgue space on this measure space. 
We utilise the partial ordering $\leq$ defined in
the standard almost everywhere sense through $\vartheta$. Take $V$ to be a separable
Hilbert space with $V \cts H$ (a continuous embedding) and the
property that $v \in V$ implies $v^+\in V$ and that there exists a $C>0$ with
$\|v^+\|_V\leq C\|v\|_V$ for all $v\in V$.
Here, $(\cdot)^+ = \max(0,\cdot)$
denotes the positive part function. Let $A \colon V \to V^*$ be a bounded,
linear, coercive  and T-monotone operator and suppose that $\Phi\colon H \to V$ is a given
obstacle map which is increasing. 
Given a source term $f \in V^*$, consider the quasi-variational inequality (QVI)
\begin{equation}
\text{find } u \in V,\; u \leq \Phi(u)
\quad\text{such that}
\quad
\langle Au-f, u -v \rangle \leq 0\quad \forall v \in V \text{ with } v \leq \Phi(u).\label{eq:QVIIntro}
\end{equation}
Under certain circumstances, this inequality has solutions that can be ordered and we denote by $\M(f)$ the maximal solution of \eqref{eq:QVIIntro} and by $\m(f)$ the minimal solution.

In this paper, we study the sensitivity and directional differentiability of these extremal solution maps $\M$ and $\m$, in addition to deriving stationarity conditions for  optimisation problems with QVI constraints of the form
\begin{equation}\label{eq:ocProblemMostGeneral}
\min_{\substack{f \in \Uad}} J(\M(f),\m(f),f).
\end{equation}
Regarding particular instances of $J$, we have in mind optimisation problems such as
\begin{equation}\label{eq:ocProblemExamples}
\min_{\substack{f \in \Uad}} \frac 12\norm{\M(f) - \m(f)}{H}^2 + \frac{\nu}{2}\norm{f}{H}^2 \qquad\qquad\text{and}\qquad\qquad \min_{\substack{f \in \Uad}} \frac 12\norm{\M(f) -y_d}{H}^2 + \frac{\nu}{2}\norm{f}{H}^2.
\end{equation}
The first is a formulation aiming to minimise the variation in solutions, first modelled and motivated in \cite{AHRStability}, and the second is the typical tracking-type problem. 

Inspired in part by our interest in deriving stationarity conditions for the control problems and in part by some results of Lions and Bensoussan in \cite[Chapter 4]{LionsBensoussan}, a substantial portion of this paper is devoted to the study of the following penalised problem associated to the QVI \eqref{eq:QVIIntro}
\begin{equation}
Au + \frac 1\rho \sigma_\rho(u-\Phi(u))  = f,\label{eq:penalisedPDEGeneralMRho}
\end{equation}
where $\rho>0$ is a parameter and $\sigma_\rho$ is the following smoothed approximation of $(\cdot)^+$
\begin{equation}\label{eq:mrhoHK}
\sigma_\rho(r)
:= \begin{cases}
0 &\text{if } r \leq 0,\\
\frac{r^2}{2\rho} &\text{if } 0 < r < \rho, \\
r-\frac{\rho}{2} &\text{if } r \geq \rho,
\end{cases}
\end{equation}
It turns out that \eqref{eq:penalisedPDEGeneralMRho} also has multiple solutions that can be ordered and we can again find a maximal solution $\M_\rho(f)$ and a minimal one $\m_\rho(f)$. We  provide a substantive analysis of the properties of these maps $\M_\rho, \m_\rho$ and also their limiting behaviour as $\rho \searrow 0$.

For convenience, we summarise our most important findings.
\begin{itemize}
\item We show that $\M_\rho(f)$ and $\m_\rho(f)$ converge to $\M(f)$ and $\m(f)$ respectively under some assumptions. Along the way, we prove that $\M_\rho(f)$ and $\m_\rho(f)$ can themselves be approximated by iterative sequences of solutions of PDEs, opening up the possibility for computation and numerical simulation (see \cref{remark:constructiveApproxOfExtremals} for details).
\item We prove that all four of these extremal solution maps ($\M_\rho, \m_\rho, \M,$ and $\m$) are locally Lipschitz from $V^*$ into $V$ (by a bootstrapping and contraction argument; we also utilise some sharp estimates from \cite{Wachsmuth2021:2} to ensure that our assumptions are kept as unobtrusive as possible). 

\item We demonstrate that the four maps are directionally differentiable for more general directions than in previous works, and also Hadamard differentiable in a certain sense (the proof is along the lines of the iterative approach of \cite{AHR} with some modifications from \cite{Wachsmuth2021:2}).

\item Using the differentiability results on $\M$ and $\m$, we  derive first-order conditions of Bouligand type for the control problem. We also derive C-stationarity conditions, which is possible thanks to the various results on $\M_\rho$ and $\m_\rho$ that we obtain (we approximate \eqref{eq:ocProblemMostGeneral} with a penalised control problem and then pass to the limit).
\end{itemize}
For precise details of all the main results, see \cref{sec:mainresults} where we present them in full. Now, let us highlight the novelty and positioning of our work among the literature.
\begin{itemize}
\item Continuity of the minimal and maximal solution maps with perturbations in an $L^\infty$-type space was first proved in \cite[Theorem 4]{AHRStability} under the structural assumption that $\lambda\Phi( u) \leq \Phi(\lambda u)$ for all $\lambda \in (0,1)$ and for $u \in H_+$. In \cite[Theorem 3.2]{ChristofWachsmuth2021:1}, Lipschitz continuity of these maps was shown, again under this setup and for source terms belonging to a subset of $L^\infty$. 

In contrast, our result shows Lipschitz stability with respect to the $V$ norm and for sources in $V^*$ (thus we do not need to restrict to the $L^\infty$ setting) and we do not require the homogeneity-type assumption on $\Phi$ (we do however ask for a local small Lipschitz assumption, see \eqref{ass:PhiSmallLipschitzZ}).

In particular, if $\Phi$ has a small Lipschitz constant around $\M(f)$, we already know that locally there is a stable (with respect to the norm in $V$) solution of the QVI \cite{AHR, Wachsmuth2019:2, AHROCQVI}, but it is not clear whether these are the maximal solutions. On the other hand, there are results \cite{AHRStability, ChristofWachsmuth2021:1} showing that the maximal solution is stable (with respect to $L^\infty$). Now, our new results show that the maximal solution is indeed $V$-stable.

\item The first work on directional differentiability for solutions of QVIs in infinite dimensions is, to the best of our knowledge, \cite{AHR} where it was shown for localised solutions and for non-negative directions. Subsequent work in \cite{Wachsmuth2019:2} and \cite{AHROCQVI} relaxed the assumptions of \cite{AHR} greatly. All three papers use a type of smallness assumption on $\Phi$ (locally) similar to the one in this paper. However, neither paper tackled the case of extremal solutions. Regarding in particular differentiability for the minimal and maximal maps, this was proved in \cite{AHRExtremals} under some sign conditions on the direction and a QVI characterisation of the derivative was given. In \cite{ChristofWachsmuth2021:1}, again in an $L^\infty$-type setting and with $\Phi$ assumed to be concave, a differentiability result for the maximal solution appears and under assumptions that entail the unique global solvability of the QVI, a characterisation of the derivative is given. 

In this work, we provide a unique QVI characterisation of the directional derivative of the minimal and maximal solution maps under a general and natural function space setting and with relatively agreeable assumptions. In contrast to the two previous works \cite{AHRExtremals, ChristofWachsmuth2021:1} on extremal solution maps, we require neither sign restrictions on the perturbation directions nor concavity or homogeneity-type assumptions on $\Phi$ nor an embedding into $L^\infty$. 

\item The study of the specifics of the maps $\M_\rho$ and $\m_\rho$ in this general setting seems to entirely new, although we should once again remind the reader that \cite{LionsBensoussan} contains some results on the convergence behaviour of these maps in a specific setting (and not in generality like ours). The results on the sensitivity and differentiability of the maps are new as are the convergence results in this generality. 

\item The stationarity conditions for the control problem involving minimal and maximal solution maps are also entirely new. The works \cite{AHROCQVI, Wachsmuth2019:2} have addressed stationarity for control problems in a QVI setting but not for the extremal solution maps. Furthermore, our C-stationarity system in some sense improves the one in \cite{AHROCQVI}	 because we are able to show that the multipliers for the adjoints vanish on the inactive set (formally speaking; see \cref{prop:multiplerConditionsGerd}) without requiring any additional strong  assumptions.

\item On this note, we are for the first time able to treat problems like the first one in \eqref{eq:ocProblemExamples} in a substantial way. 

\item Our results remain valid when the obstacle mapping $\Phi \equiv \psi$ is constant, i.e., in the case where \eqref{eq:QVIIntro} is a variational inequality. We note in particular that \cref{prop:multiplerConditionsGerd} improves the $\mathcal{E}$-almost conditions derived in \cite[Theorem 3.4]{MR2822818} for control of the obstacle problem; see \cref{prop:epsAlmostCStationarity}.
\end{itemize}

Although we have specified the functional framework of this paper with the base space chosen as $L^2(\Omega)$, let us stress that in fact, many of our results will apply in far greater generality, with a much more general function space setting (than $H=L^2(\Omega)$ as taken above) and also with far more general maps $\sigma_\rho$ (provided certain crucial properties are satisfied) than the one above, such as for example $\sigma_\rho(u) := u^+$. For simplicity and clarity of exposition, we have decided to present our work with the choice of $H$ as above and with $\sigma_\rho$ as in \eqref{eq:mrhoHK} in the paper. We will not present the details here but invite interested readers to work out the details.

Regarding the organisation of the paper, we begin in \cref{sec:basicAssumptions} with some basic definitions, notations and fundamental results. In \cref{sec:mainresults} we state all of our main results for the convenience of the reader, include also some useful or interesting remarks and provide in \cref{sec:example} some examples. In \cref{sec:penalisedProblem}, we study \eqref{eq:penalisedPDEGeneralMRho} and an iterative sequence of associated problems and show that \eqref{eq:penalisedPDEGeneralMRho} does indeed possess extremal solutions. \cref{sec:convergencetoQVIs} is devoted to the study of the limit $\rho \to 0$ in \eqref{eq:penalisedPDEGeneralMRho}, both with and without a locally small Lipschitz assumption on $\Phi$. Using these obtained results, we prove our claims on the Lipschitzness of all the maps in \cref{sec:lipschitz} and directional differentiability in \cref{sec:dirDiff}. In \cref{sec:oc}, we study the optimisation problem \eqref{eq:ocProblemMostGeneral} and prove B-stationarity and various forms of C-stationarity.  In \cref{sec:conclusion}, we finish the main part of the paper with some final remarks.

\subsection{Notation and preliminaries}\label{sec:basicAssumptions}

Define the set $H_+ := \{ h \in H : h \geq 0\}$ of non-negative elements of $H = L^2(\Omega)$ and similarly, define $V_+$. We write $h^+=P_{H_+}h$ to denote the orthogonal projection of $h \in H$ onto $H_+$ and define $h^-:=h^+-h$. The infimum and supremum of two elements $h_1,h_2\in H$ are defined as usual: $\inf (h_1,h_2):=h_1-(h_1-h_2)^+$ and $\sup (h_1,h_2):=h_1+(h_2-h_1)^+$.
We define an order on the dual space $V^*$ via
\[ f \leq g \iff \langle f-g, v \rangle  \leq 0 \quad \forall v \in V_+.\]
Here $\langle \cdot, \cdot \rangle$ is the duality pairing between $V^*$ and $V$. Regarding the elliptic operator in \eqref{eq:QVIIntro}, as mentioned, we take $A\colon V \to V^*$ to be a linear operator that satisfies the following properties for all $u, v \in V$:
\begin{align*}
\langle Au, v \rangle &\leq C_b\norm{u}{V}\norm{v}{V}\tag{boundedness},\\
\langle Au, u \rangle &\geq C_a\norm{u}{V}^2\tag{coercivity},\\
\langle Au^+, u^- \rangle &\leq 0\tag{T-monotonicity},
\end{align*}
where $C_a, C_b > 0$ are constants.

With $\mathbf{K}(u) := \{ v \in V : v \leq \Phi(u)\}$, the QVI \eqref{eq:QVIIntro} can be written as	
\begin{equation*}
u \in \mathbf{K}(u) : \langle Au - f, u-v \rangle \leq 0 \quad \forall v \in \mathbf{K}(u).
\end{equation*}
We introduce $S\colon V^* \times H \to V$ as the solution map of the associated variational inequality, i.e, $u=S(f,\psi)$ if and only if
\[u \in \mathbf{K}(\psi) : \langle Au-f, u- v \rangle \leq 0 \quad \forall v \in \mathbf{K}(\psi).\]
Thus the solutions of \eqref{eq:QVIIntro} are precisely the fixed points of $S(f,\cdot)$.
\begin{ass} \label{ass:forZSubAndSuperSolutions}

Given $f \in V^*$, assume that there exist $\underline u, \overline u \in V$ such that
\[\underline u \leq S(f, \underline u), \quad \overline{u} \geq S(f, \overline u), \quad\text{and}\quad  \underline u \leq \overline u.\]
\end{ass}
The element $\underline u$ is called a \emph{subsolution} for $S(f,\cdot)$ and $\overline u$ is called a \emph{supersolution} for $S(f,\cdot)$. 

We come now to an existence result for \eqref{eq:QVIIntro}. For more existence results under different assumptions, see \cite[\S 2]{AHROCQVI}. 

\begin{prop}\label{prop:existenceOfZ}Under \cref{ass:forZSubAndSuperSolutions}, there exist a minimal solution $\m(f)$ and maximal solution $\M(f)$ to \eqref{eq:QVIIntro} on
	the interval $[\underline u, \overline u] := \set{v \in V \given \underline u \le v \le \overline u \text{ a.e.\ in } \Omega}$.
\end{prop}
\begin{proof}
We apply the Birkhoff--Tartar theorem \cite[\S 15.2.2, Proposition 2]{MR556865} which gives existence of fixed points for increasing maps that possess subsolutions and supersolutions to the map $S(f,\cdot)$ (which is increasing, see \cite[\S 4:5, Theorem 5.1]{Rodrigues}). See also \cite{tartar1974inequations}, \cite[\S 11.2]{MR745619} and \cite[Chapter 2]{Mosco}.

\end{proof}
We will use the notation $B_r(x)$ to denote the (closed) ball of radius $r$ centred at $x$. It should be clear from the context the function space in which the ball is taken but typically when we use $\delta$ (or a variant such as $\bar \delta$) for the radius, it refers to the $V^*$ ball, whereas the radius being $\epsilon$ (or a variant) refers to the $V$ ball.

\section{Main results}\label{sec:mainresults}
Let us discuss our main results. As a matter of notation, to handle both cases (of the minimal and maximal solution maps) simultaneously, we denote by $\Z$ one of the maps $\M$ or $\m$. Note that $\Z$ is defined at all points $f$ satisfying \cref{ass:forZSubAndSuperSolutions}.

\subsection{On directional differentiability}
Our first result concerns local Lipschitz continuity of $\Z$. For this, we need $\Z$ to be defined not just at a solitary point but in a neighbourhood. Thus, we need to expand \cref{ass:forZSubAndSuperSolutions} to take this into account.
\begin{ass}\label{ass:subAndSupersolutionsZ}
Let $f \in V^*$ and take a set  $W \subseteq V^*$ containing $f$ and assume  that there exist $\underline u, \overline u \in V$ and $\bar \delta$ such that
\begin{subequations}\label{ass:aAssumptionsForVI}
\begin{alignat}{2}
\underline u &\leq \overline u,\\
\underline u &\leq S(g, \underline u) \qquad &&\forall g \in B_{\bar\delta}(f) \cap W,\label{ass:a1VI}\\
\overline u &\geq S(g,\overline u) &&\forall g \in B_{\bar\delta}(f) \cap W\label{ass:a2VI}.
\end{alignat}
\end{subequations}
\end{ass}
The intersection with the set $W$ that appears in the assumption above is inspired by applications where the source terms may lie in some given ordered interval and it ensures that natural candidates for the sub- and supersolutions (namely those arising from the boundary of the ordered interval) indeed qualify as sub- and supersolutions, see the next remark.

\begin{remark}\label{remark:whyWeUseW}
Consider the example in \cref{sec:example}. If we had asked for  \eqref{ass:aAssumptionsForVI} to hold for all $g \in B_{\bar \delta}(f)$ (i.e. without the intersection with a set $W$), then $\overline u = 0$ does not satisfy \eqref{ass:a1VI} for the element $f=0 \in V^*$ since $B_{\bar \delta}(0)$ contains negative functions, so that $0 \leq S(g,0)$ may not hold for all $g \in B_{\bar \delta}(0)$. {Even worse, due to $S(g, \underline u) \le A^{-1} g$ for all $g \in V^*$ and $\underline u \in V$,
we would need $\underline u \le A^{-1} g$ for all $g \in B_{\bar\delta}(f)$, but this is not possible since $A^{-1} g$ could have negative singularities at arbitrary points.}
Hence, the intersection with $W$ is necessary for the existence of sub- and supersolutions.
\end{remark}
The next theorem will be proved in \cref{sec:lipschitz}. 
\begin{theorem}[Local Lipschitz continuity of $\Z$]\label{thm:localLipschitzForZ}
Let $f \in V^*$  and $W \subseteq V^*$ satisfy \cref{ass:subAndSupersolutionsZ}. Assume
\begin{align}
	&\Phi\colon V \to V \text{ is completely continuous}\label{ass:PhiCC},\\
\nonumber	&\text{there exists $\epsilon^* > 0$ such that $\Phi\colon B_{\epsilon^*}(\Z(f)) \to V$ has a Lipschitz constant $C_L$ satisfying} \\
								&\quad \text{$C_L < \frac{C_a}{C_b}$ or $A$ is self-adjoint and $C_L < 2\frac{ \sqrt{C_b/C_a}}{1 + C_b/C_a}$.}\label{ass:PhiSmallLipschitzZ}
\end{align}
Then then there exists  $\delta \in (0, \bar\delta)$ such that for all $g \in B_{\delta}(f) \cap W$,
\[\norm{\Z(f)-\Z(g)}{V} \leq C\norm{f-g}{V^*}\]
where $C > 0$ is a constant (which depends only on $C_L$, $C_a$, $C_b$ and the self-adjointness of $A$).
\end{theorem}
In the assumption \eqref{ass:PhiSmallLipschitzZ} above, `self-adjoint' essentially means that the associated bilinear form is symmetric. Note that \eqref{ass:PhiSmallLipschitzZ} is indeed a rather strong assumption as it asks for a smallness condition on the Lipschitz constant of $\Phi$, albeit only locally. In particular, the assumption implies local uniqueness on the ball $B_{\epsilon^*}(\Z(f))$: any solution which exists in the ball is isolated. The nature of QVIs where non-uniqueness appears seems to necessitate such assumptions. We will demonstrate in \cref{sec:thermoforming} a real-world application in which such an assumption is satisfied.

With the addition of just one more assumption (namely the differentiability of $\Phi$ at a point) we can secure directional differentiability. Before we state the result, let us recall that the \emph{radial cone} of a set $C \subset X$ of a Banach space $X$ at a point $x \in C$ is defined as
\[\mathcal{R}_C(x) := \{ y \in X \mid \exists s_0 > 0 : x+sy \in C \;\; \forall s \in [0, s_0]\}.\]
The \emph{tangent cone} is defined as
\[\mathcal{T}_C(x) := \{ y \in X \mid \exists s_k \searrow 0, \; \exists y_k \to y \text{ in } X : x + s_ky_k \in C \;\;  \forall k \}.\]
In the case that $C$ is additionally convex,
the tangent cone is the closure of the radial cone in $X$, written $\mathcal{T}_C(x) = \overline{\mathcal{R}_C(x)}$.

\begin{theorem}[Hadamard differentiability of $\Z$]\label{thm:hadamardDiffZNEW}
Let $f \in V^*$ and $W \subseteq V^*$ satisfy \cref{ass:subAndSupersolutionsZ}. Assume  
\eqref{ass:PhiCC}, \eqref{ass:PhiSmallLipschitzZ} and 
\begin{align}
&\text{$\Phi$ is directionally differentiable at $\Z(f)$}\label{ass:PhiDiffAtUNew}.
\end{align}
Then
\begin{enumerate}[label=(\roman*)]\itemsep=0cm
\item\label{item:ZIsHadamardDiff} the map $\Z$ is Hadamard differentiable in the sense that if $d \in \mathcal{T}_W(f)$, then for any sequence $d_k \to d$ in $V^*$ with $f+s_kd_k \in W$ where $s_k \searrow 0$,
\[ \frac{\Z(f+s_kd_k)-\Z(f)}{s_k} \to \Z'(f)(d).\]
\item the derivative $\Z'(f)(d)$ is the unique solution of the QVI
	 \begin{equation}
 		\alpha \in \mathcal{K}^u(\alpha) : \langle A\alpha -d, \alpha - v \rangle \leq 0 \quad \forall v \in \mathcal{K}^u(\alpha)	 \label{eq:inequalityForAlpha}
	 \end{equation}
	 where, writing $u=\Z(f)$,
	 \[\mathcal{K}^u(\alpha) := \Phi'(u)(\alpha) + \mathcal{T}_{\mathbf K(u)}(u) \cap [f-Au]^\perp.\]
\item the map $\Z'(f)\colon \mathcal{T}_W(f) \to V$ can be extended to a bounded and continuous mapping from $V^*$ to $V$ by defining it via \eqref{eq:inequalityForAlpha} for all $d \in V^*$.  

\end{enumerate}
\end{theorem}
For the proof, see \cref{sec:differentiabilityForZProof}.
\begin{remark}[Directional differentiability of $\Z$]\label{rem:HadamardToDirDiff}
A simple corollary of \cref{thm:hadamardDiffZNEW} \ref{item:ZIsHadamardDiff} is that the map $\Z\colon B_{\bar\delta}(f) \cap W \to V$ is directionally differentiable at $f$ in every direction $d \in \mathcal{R}_W(f) \subset V^*$:
\[\lim_{s \searrow 0}\frac{\Z(f+sd) - \Z(f)}{s} = \Z'(f)(d).\]	
Taking the direction from $\mathcal{R}_W(f)$ ensures that the perturbed solution $\Z(f+sd)$ is well defined via \cref{ass:subAndSupersolutionsZ}.
\end{remark}
\begin{eg}[The radial cone $\mathcal{R}_W(f)$]\label{eg:LinftyRadialCone}
Similarly to \cref{sec:example}, let us consider
\[W := \{ g \in V^* : 0 \leq g \leq F\}\]
with $F \geq k_0$ for some constant $k_0 > 0$, and $\underline u := 0$ and $\overline u := A^{-1}F$. Let us try to describe the radial cone at different points in $W$.
\begin{itemize}
\item Take $d \in L^\infty_+(\Omega)$. Then $sd \geq 0$ for all $s>0$ and if $s \leq k_0\slash \norm{d}{L^\infty(\Omega)}$, we have, for all $\varphi \in V_+$,
\[\langle sd-F, \varphi \rangle = \langle sd-k_0, \varphi \rangle + \langle k_0 - F, \varphi \rangle \leq 0,\]
so that $sd \in W$ for sufficiently small $s$. This shows that $L^\infty_+(\Omega) \subset \mathcal{R}_W(0)$.

\item In a similar way, take $d \in L^\infty_-(\Omega)$. For all $s \geq 0$, we have $F+sd \leq F$ and if $s \leq k_0\slash \norm{d}{L^\infty(\Omega)}$, we have
\[\langle F+sd,\varphi \rangle = \langle F-k_0, \varphi \rangle + \langle k_0 + sd, \varphi \rangle\]
and $k_0+sd \geq k_0 + k_0d\slash \norm{d}{L^\infty(\Omega)}  = k_0(1+d\slash \norm{d}{L^\infty(\Omega)})\geq 0$, thus $F+sd \geq 0$ and we have shown that $L^\infty_-(\Omega) \subset \mathcal{R}_W(F)$.

\item Now consider a point $f$ such that $0 < c_0 \leq f \leq c_1 < F$ where $c_0$ and $c_1$ are constants.
	Using similar arguments as above, we can show
	$L^\infty(\Omega) \subset \mathcal{R}_W(f)$.
\end{itemize}
\end{eg}

\subsection{On the penalised problem}\label{sec:main_results_penalised}
In this section, we address results for the 
penalised problem \eqref{eq:penalisedPDEGeneralMRho}, i.e.,
\begin{equation*}
Au+ \frac{1}{\rho}\sigma_\rho(u-\Phi(u)) = f.
\end{equation*} 
We denote by $T_\rho\colon V^* \times H \to V$ the solution map $(f,w) \mapsto u$ of the corresponding equation
\[Au + \frac{1}{\rho}\sigma_\rho(u-\Phi(w)) = f.\]
Here, we associate with the real-valued function $\sigma_\rho$ defined in \eqref{eq:mrhoHK} the operator $\sigma_\rho\colon V \to V^*$ defined as
\[\langle \sigma_\rho(u), v \rangle = \int_\Omega \sigma_\rho(u)v.\]
If, given $f \in V^*$ and a fixed $\rho > 0$, we have the availability of $\underline u, \overline u$ such that
\begin{equation}\label{ass:newTRhoSubSup}
\underline u \leq T_\rho(f, \underline u), \quad \overline{u} \geq T_\rho(f, \overline u), \quad\text{and}\quad  \underline u \leq \overline u,
\end{equation}
then  there exist a minimal solution $\m_\rho(f)$ and maximal solution $\M_\rho(f)$ to \eqref{eq:penalisedPDEGeneralMRho}
on $[\underline u, \overline u]$. We will show this in \cref{prop:existenceOfZRho}. In a similar way to before,
we use $\Z_\rho$ to denote either $\M_\rho$ or $\m_\rho$.

Since we want to consider the limit $\rho \searrow 0$, we need $\Z_\rho$ to be defined for sufficiently small $\rho$, hence \eqref{ass:newTRhoSubSup} (which holds for a fixed $\rho$) needs to be modified. We do this in the next assumption, which kills two birds with one stone: it also ensures that both $\Z_\rho$ and $\Z$ are defined on a neighbourhood (just like we argued for \cref{ass:subAndSupersolutionsZ}) and not just at one point.
\begin{ass}\label{ass:subAndSupersolutionsForZRho}Let $f \in V^*$ and take a set  $W \subseteq V^*$ containing $f$ and assume that there exist $\underline u, \overline u \in V$ and $\bar \delta, \rho_0 >0$ such that
\begin{subequations}\label{ass:aAssumptions}
\begin{alignat}{2}
\underline u &\leq \overline u,\\
\underline u &\leq S(g, \underline u) \qquad &&\forall g \in B_{\bar\delta}(f) \cap W,\label{ass:a1}\\
\overline u &\geq T_{\rho_0}(g,\overline u) &&\forall g \in B_{\bar\delta}(f)\cap W.\label{ass:a2}
\end{alignat}
\end{subequations}
\end{ass}
The fundamental question is whether $\M_\rho(f)$ and $\m_\rho(f)$ converge (in some sense) to $\M(f)$ and $\m(f)$. In fact, we can even prove something stronger with the following joint (in $\rho$ and the source term) continuity result; its proof appears in \cref{sec:convUnderContraction}.
\begin{theorem}[Convergence of $\Z_\rho(g)$ to $\Z(f)$]\label{thm:jointContinuityOfRhoMapsUnderContraction}
Let $f \in V^*$ and $W \subseteq V^*$ satisfy \cref{ass:subAndSupersolutionsForZRho}. 
Assume
\eqref{ass:PhiCC}
\footnote{Instead of \eqref{ass:PhiCC}
we could assume that $\Phi\colon V \to V$ is continuous, \eqref{ass:PhiWeakSeqCts}, \eqref{ass:VCompactInH} and \eqref{ass:gerd1}.}
and
\eqref{ass:PhiSmallLipschitzZ}.
Then
		\begin{equation*}
			\lim_{\substack{\rho \searrow 0\\g\to f}} \Z_\rho(g)		=
			\Z(f)
		\end{equation*}
		where the convergence $g \to f$ is understood in $V^*$ and for $g \in W$.
\end{theorem}
As we said in \cref{remark:whyWeUseW}, having $W \neq V^*$ in \cref{ass:subAndSupersolutionsForZRho} above makes it a weaker assumption than if it held with $W$ equal to the entire space $V^*$, and leads to a convergence result with respect to $g$ that is perhaps weaker than one might first expect, but this is obviously natural since the extremal maps only exist for such source terms.

By choosing $W = \{f\}$ in the statement of the theorem, we
get the corollary below. Note that the assumption below essentially asks for
the inequalities in \eqref{ass:aAssumptions} to hold only at ($g$ replaced
with) the particular point $f$.
\begin{cor}[Convergence of $\Z_\rho(f)$ to $\Z(f)$]\label{thm:continuityOfRhoMapsUnderContraction}
Let  $f \in V^*$ and $W:=\{f\}$ satisfy \cref{ass:subAndSupersolutionsForZRho} and assume
\eqref{ass:PhiCC}
and \eqref{ass:PhiSmallLipschitzZ}.
Then $\Z_\rho(f) \to \Z(f)$ in $V$.
\end{cor}
\begin{remark}
This result uses the small Lipschitz condition \eqref{ass:PhiSmallLipschitzZ} but it is not necessary to obtain the convergence of $\M_\rho(f)$ to $\M(f)$, see \cref{thm:convergenceOfMRho}. It is an open problem whether $\m_\rho(f)$ converges to $\m(f)$ under the general assumptions of \cref{thm:convergenceOfMRho}. 
\end{remark}

In a similar fashion to \cref{thm:localLipschitzForZ} and \cref{thm:hadamardDiffZNEW}, we have the following local Lipschitz and differentiability results for $\Z_\rho$, proven in \cref{sec:lipschitz} and \cref{sec:rhoPositiveCase} respectively.
\begin{theorem}[Local Lipschitz continuity of $\Z_\rho$]\label{thm:localLipschitzForZRho}
Let $f \in V^*$ and $W \subseteq V^*$ satisfy \cref{ass:subAndSupersolutionsForZRho}. Assume
\eqref{ass:PhiCC} and
\eqref{ass:PhiSmallLipschitzZ}.
Then there exist $\rho_0$ and $\delta  > 0$ such that for all $\rho \leq \rho_0$ and $g \in B_{ \delta }(f) \cap W$,
\[\norm{\Z_\rho(f)-\Z_\rho(g)}{V} \leq C\norm{f-g}{V^*}\]
where $C > 0$ is a constant (which depends only on $C_L$, $C_a$, $C_b$ and the self-adjointness of $A$).
\end{theorem}

\begin{theorem}[Hadamard differentiability of $\Z_\rho$]\label{thm:hadamardDiffMRhoNEW}
Let $f \in V^*$ and $W \subseteq V^*$ satisfy \cref{ass:subAndSupersolutionsForZRho}. Assume
\eqref{ass:PhiCC},
\eqref{ass:PhiSmallLipschitzZ} and 
\begin{align}
&\text{$\Phi$ is directionally differentiable at $\Z_\rho(f)$}\label{ass:PhiDiffAtU}.
\end{align}
Then  for  $\rho$ sufficiently small,
\begin{enumerate}[label=(\roman*)]\itemsep=0cm
\item\label{item:HadamardDifferentiabilityOfZRho} the map $\Z_\rho$ is Hadamard differentiable in the sense that if $d \in \mathcal{T}_W(f)$, then for any sequence $d_k \to d$ in $V^*$ with $f+s_kd_k \in W$ where $s_k \searrow 0$,
\[ \frac{\Z_\rho(f+s_kd_k)-\Z_\rho(f)}{s_k} \to \Z_\rho'(f)(d).\]
\item the derivative $\Z_\rho'(f)(d)$ is the unique solution of the equation
	 \begin{equation}
 		A\alpha + \frac{1}{\rho}\sigma_\rho'(u-\Phi(u))(\alpha-\Phi'(u)(\alpha)) = d \label{eq:equationForAlphaRhoNEW}
	 \end{equation}
where $u=\Z_\rho(f)$.	
\item the map $\Z_\rho'(f)\colon \mathcal{T}_W(f) \to V$ can be extended to a bounded and continuous mapping from $V^*$ to $V$ by defining it via \eqref{eq:equationForAlphaRhoNEW} for all $d \in V^*$.  
\end{enumerate}
\end{theorem}
Exactly as in \cref{rem:HadamardToDirDiff}, we obtain from \cref{thm:hadamardDiffMRhoNEW} \ref{item:HadamardDifferentiabilityOfZRho} the directional differentability of $\Z_\rho\colon B_{\bar \delta}(f) \cap W \to V$ at $f$ in every direction $d \in \mathcal{R}_W(f) \subset V^*$:
	\[\lim_{s \searrow 0}\frac{\Z_\rho(f+sd) - \Z_\rho(f)}{s} = \Z_\rho'(f)(d).\]

\subsection{On optimal control}
Regarding existing literature on the derivation of stationarity systems for optimal control with QVI constraints, we mention \cite{Wachsmuth2019:2} and \cite{AHROCQVI} in particular. The first work contains a strong stationarity system characterisation in the absence of control constraints, whilst the latter work includes the derivation of various forms of stationarity systems (including strong) with potential box constraints on the control. In this work, we extend these results to the setting of minimal and maximal solution mappings and derive a C-stationarity system.

Suppose that
\begin{equation*}
\text{$V \ctsCompact H \cts V^*$ is a Gelfand triple}
\end{equation*} 
($\ctsCompact$ means a compact embedding;  by definition of the Gelfand triple, $V \ctsDense H$ is a dense embedding) and let $\Uad \subset H$ be a non-empty, closed and convex set\footnote{It would suffice to replace `closed and convex' here with `weakly sequentially closed' (which is a weaker requirement) for the existence results  below.}. Recall the control problem \eqref{eq:ocProblemMostGeneral}, reproduced here:
\begin{equation*}
\min_{\substack{f \in \Uad}} J(\M(f),\m(f),f).
\end{equation*}
 We make the next standing assumption, which guarantees the well definedness of \eqref{eq:ocProblemMostGeneral}.
\begin{ass}\label{ass:ocSubAndSupersolutions}
There exist $\underline u, \overline u \in V$ such that
\begin{equation*}
\begin{aligned}
\underline u &\leq \overline u,\\
\underline u &\leq S(g, \underline u) &&\forall g \in \Uad,\\
\overline u &\geq S(g, \overline  u) &&\forall g \in \Uad.\\
\end{aligned}
\end{equation*}
\end{ass}

Regarding the objective functional $J$, we need the following assumptions in place. Observe that the last two assumptions below are conditions that involve $\Phi$.
\begin{ass}\label{ass:smallLipschitzForOCProblem}
Regarding $J(y,z,f)$, assume that 
\begin{enumerate}[label=(\roman*)]\itemsep=0cm
\item\label{item:test} $J\colon V \times V \times H \to \mathbb{R}$ is continuously Fr\'echet differentiable and bounded from below.
\item\label{item:wlscForJ} If $(y_n, z_n) \to (y,z)$ in $V\times V$ and $f_n \weaklyto f$ in $H$, then 
\[J(y,z,f) \leq \liminf_{n \to \infty} J(y_n, z_n, f_n).\]
\item\label{item:uniformBd}
	If
	$\{J(y_n,z_n,f_n)\}$ is bounded
	for a sequence $\{(y_n, z_n, f_n)\} \subset V\times V \times \Uad$,
	then $\{f_n\}$ is bounded in $H$.

\item \label{item:smallLipschitzM}  If $J_y \not\equiv 0$, for every $f \in \Uad$, there exists ${\epsilon^*} > 0$ such that $\Phi\colon B_{\epsilon^*}(\M(f)) \to V$ has a Lipschitz constant satisfying 			$C_L < {C_a}\slash {C_b}$ or $A$ is self-adjoint and $C_L < 2{ \sqrt{C_b/C_a}}(1 + C_b/C_a)^{-1}$.
\item \label{item:smallLipschitzm} If $J_z \not\equiv 0$, for every $f \in \Uad$, there exists ${\epsilon^*} > 0$ such that $\Phi\colon B_{\epsilon^*}(\m(f)) \to V$ has a Lipschitz constant satisfying $C_L < {C_a}\slash {C_b}$ or $A$ is self-adjoint and $C_L < 2{ \sqrt{C_b/C_a}}(1 + C_b/C_a)^{-1}$.
\end{enumerate}
\end{ass}
An example of $J$ satisfying items \ref{item:test}--\ref{item:uniformBd} above is  
\begin{equation*}
J(y,z,f) =  \frac 12\norm{ay + bz-y_d}{H}^2 + \frac{\nu}{2}\norm{f}{H}^2
\end{equation*}
given constants $a, b \in \mathbb{R}$, $\nu > 0$, and for some given $y_d \in H$. When we choose $a=1$, $b=-1$ and $y_d \equiv 0$, we recover the first  objective functional in \eqref{eq:ocProblemExamples} and  when $a=1$ and $b=0$ or vice versa, we recover the second one in \eqref{eq:ocProblemExamples}.

We remark that the assumptions in items \ref{item:smallLipschitzM} and \ref{item:smallLipschitzm} are unfortunately rather unsatisfactory because they impose local uniqueness around the extremal solutions for \textit{every} source term in $\Uad$.

\begin{theorem}[Existence of optimal controls]\label{thm:rfsfsdd}
Assume \eqref{ass:PhiCC}, \cref{ass:ocSubAndSupersolutions} and \cref{ass:smallLipschitzForOCProblem}.
Then there exists an optimal control $f^* \in \Uad$ to the problem \eqref{eq:ocProblemMostGeneral}.
\end{theorem}
The proof (see \cref{sec:oc}) is more or less standard and uses the direct method in the calculus of variations.
From now on, let 
\[\text{$(y^*, z^*, f^*)$ be an arbitrary local minimiser of \eqref{eq:ocProblemMostGeneral}}\]
with $y^* = \M(f^*)$ and $z^*=\m(f^*)$. We begin with the following primal characterisation of the minimiser.

\begin{prop}[Bouligand stationarity]\label{lem:characterisationOfOC}Assume \eqref{ass:PhiCC}, \cref{ass:ocSubAndSupersolutions},  \cref{ass:smallLipschitzForOCProblem}
 and  
 \begin{align*}
& \text{if $J_y \not\equiv 0$, $\Phi$ is directionally differentiable at $\M(f^*)$},\\
& \text{if $J_z \not\equiv 0$, $\Phi$ is directionally differentiable at $\m(f^*)$}.
 \end{align*}
Then  
\begin{equation*}
\langle J_y(y^*,z^*,f^*), \M'(f^*)(h)\rangle + \langle J_z(y^*,z^*,f^*), \m'(f^*)(h)\rangle + \langle J_f(y^*,z^*,f^*), h \rangle  \geq 0 \quad \forall h \in \mathcal{T}_{\Uad}(f^*).
\end{equation*} 
\end{prop}
The proof of the proposition appears in \cref{sec:BS}.

For numerics, it is convenient to derive other forms of stationarity systems like C-stationarity. For this purpose, we consider the penalised control problem 
\begin{equation}
\min_{f \in \Uad}J(\M_\rho(f), \m_\rho(f), f).\label{eq:ocProblemMainResults}
\end{equation}
The following standing assumption is stronger than \cref{ass:ocSubAndSupersolutions} and it implies the assumptions of \cref{thm:localLipschitzForZRho}, which is needed for the existence of controls for the above control problem.
\begin{ass}\label{ass:ocSubAndSupersolutionsStronger}
There exist $\underline u, \overline u \in V$ and $\rho_0 > 0$ such that
\begin{equation*}
\begin{aligned}
\underline u &\leq \overline u,\\
\underline u &\leq S(g, \underline u) &&\forall g \in \Uad,\\
 \overline u &\geq T_{\rho_0}(g, \overline u) &&\forall g \in \Uad.
\end{aligned}
\end{equation*}
\end{ass}
We need some further regularity on $\Phi$ in the form of the next assumption. When $\Phi$ is continuously  Fr\'echet differentiable, these assumptions follow from \cref{ass:smallLipschitzForOCProblem} \ref{item:smallLipschitzM}, \ref{item:smallLipschitzm}. See the discussion around \eqref{ass:newAPhiInvCoerciveUniform} and the proof of \cite[Lemma 5.9]{AHROCQVI}.

\begin{ass}\label{ass:consolidatedAss}
Assume the following:
\begin{enumerate}[label=(\roman*)]\itemsep=0cm
\item\label{item:PhiDirDiffAndDerivLinearInBall} If $J_y \not\equiv 0$, assume that  there exists $\epsilon>0$ such that 
\begin{equation*}
\text{for all $w \in B_\epsilon(y^*)$, $\Phi$ is directionally differentiable at $w$ and $\Phi'(w)$ is linear}.
\end{equation*}
If $J_z \not\equiv 0$, the above holds with $y^*$ replaced by $z^*$.
\item\label{item:doubleConts} If $J_y \not\equiv 0$, assume that for sequences $v_n \to v$, $w_n \to w$ and $q_n \weaklyto q$  in $V$ with $v_n, v \in B_\epsilon(y^*)$, we have
\begin{align}
(\Id-\Phi'(v_n))^{-1}q_n &\weaklyto (\Id-\Phi'(v))^{-1}q \text{ in $V$}\label{ass:doubleContWeak},\\
(\Id-\Phi'(v_n))^{-1}w_n &\to (\Id-\Phi'(v))^{-1}w \text{ in $V$}\label{ass:doubleCont},
\end{align}
If $J_z \not\equiv 0$, the above holds with $y^*$ replaced by $z^*$.
\end{enumerate}
\end{ass}

We also need additional structure on the function spaces in the form of a Dirichlet space. 
\begin{ass}\label{ass:forCStationarity}
Let $V$ be a regular Dirichlet space and suppose that $(\cdot)^+ \colon V \to V$ is continuous.
\end{ass}
We will not enter into an exposition about Dirichlet spaces here (see \cite[Example 3.5]{AHROCQVI} for a convenient definition and comments on this as well as further references) but let us give some examples that satisfy the above assumption.  Suppose that $D \subset \mathbb{R}^n$ is a bounded Lipschitz domain. We can take $H=L^2(D)$ and $V=H^1_0(D)$ (thus $\Omega \equiv D$), or $H=L^2(\overline D)$ and $V = H^1(D)$ (thus $\Omega \equiv \overline D$). Fractional spaces are also a possibility. Indeed,  $V=H^s(D)$ for $s \in (0,1)$ and $H=L^2(\overline{D})$ is valid (i.e.,  $\Omega := \overline{D}$), where the fractional Sobolev space $H^s(D)$ is defined as usual as the space of measurable functions $u\colon D \to \mathbb{R}$ such that the norm
\begin{equation}\label{eq:normHs}
\norm{u}{H^s(D)} := \left(\int_{D} u^2 + \int_{D}\int_{D}\frac{|u(x)-u(y)|^2}{|x-y|^{n+2s}}\right)^{\frac 12}
\end{equation}
is finite.  On the plane, we could also pick $V=H^{s}(\mathbb{R}^d)$ and $H=L^2(\mathbb{R}^d)$ ($V$ is defined similarly to above via \eqref{eq:normHs} but with $D$ replaced with $\mathbb{R}^d$); thus here  $\Omega = \mathbb{R}^d.$ In these two cases the natural operator to choose for $A$ would be the fractional Laplacian $(-\Delta)^s$. We refer to \cite[\S 1.2.1]{AHR} for a QVI example involving the fractional Laplacian in an application involving fluid flow.  

Assuming a Dirichlet space structure enables us to define notions of capacity and quasi-continuity (capacity is, loosely speaking, a way to measure sets finer than through the Lebesgue measure), see \cite[\S 2.1]{FukushimaBook} or \cite[Section 2]{Wachsmuth2014:2} for the $H^1_0(\Omega)$ setting. In addition, it allows us to explicitly characterise the critical cone appearing in \cref{thm:hadamardDiffZNEW} (see \cref{rem:criticalConeExpression}) using capacity, and (more pertinently for us in this section) the tangent cone as well, which is something we will use to prove a statement in the stationarity system below. On that topic, note for any $y \in V$ we can define\footnote{Strictly speaking, every $y \in V$ has a quasi-continuous representative and we identify it with its representative. Then the set $\{y=\Phi(y)\}$ is quasi-closed. }  
\[\{y = \Phi(y) \} \equiv \{ x \in \Omega : y(x) = \Phi(y)(x) \},\]
which, when $y$ is a solution of the QVI, is called the \emph{active} or \emph{coincidence set}. This set is defined up to sets of capacity zero.

We will prove a version of C-stationarity (but note that this terminology is used somewhat inconsistently in the literature). Before we proceed, let us record that owing to the complementarity characterisation of solutions of QVIs (see e.g. \cite[Proposition 2.1]{AHROCQVI}), the statements $y^*=\M(f^*)$ and $z^*=\m(f^*)$ imply (but are not necessarily equivalent to) that
\begin{equation}\label{eq:complementarity}
\begin{aligned}
Ay^* - f^* +  \xi_1^* &= 0,\\
Az^* - f^* +  \xi_2^* &= 0,\\
\xi_1^* \geq 0 \text{ in $V^*$}, \quad y^* \leq \Phi(y^*), \quad \langle \xi_1^*, y^*-\Phi(y^*)\rangle &= 0,\\
\xi_2^* \geq 0 \text{ in $V^*$}, \quad z^* \leq \Phi(z^*), \quad \langle \xi_2^*, z^*-\Phi(z^*)\rangle &= 0.
\end{aligned}
\end{equation}
The main result in this section is the following, which will be proved through a succession of results in \cref{sec:CStationarity}.
\begin{theorem}[C-stationarity]\label{thm:ocPenalisation}
Assume \eqref{ass:PhiCC}, \cref{ass:smallLipschitzForOCProblem},
\cref{ass:ocSubAndSupersolutionsStronger}, \cref{ass:consolidatedAss}, and    \cref{ass:forCStationarity}. Take any local minimiser $(y^*,z^*,f^*)$ of \eqref{eq:ocProblemMostGeneral} and define $\xi_1^*, \xi_2^*$ as in \eqref{eq:complementarity}. 
Then there exist multipliers $(p^*, q^*,    \lambda^*, \zeta^*) \in   V \times V \times V^* \times V^*$ satisfying the C-stationarity system
\begin{subequations}\label{eq:cStationaritySystem}
\begin{align}
y^* &= \M(f^*),\\
z^* &= \m(f^*),\\
A^*p^* + (\Id-\Phi'(y^*))^*\lambda^* &= -J_y(y^*,z^*,f^*),\label{eq:eafirst}\\
A^*q^* + (\Id-\Phi'(z^*))^*\zeta^* &= -J_z(y^*,z^*,f^*),\label{eq:eafirstB}\\
f^* \in \Uad : \langle J_f(y^*, z^*, f^*)- p^*-q^* , f^*-v\rangle &\leq 0
 \quad \forall v \in \Uad,\\
\langle \lambda^*, p^* \rangle &\geq 0,
\label{eq:eaCs}\\
\langle \zeta^*, q^* \rangle &\geq 0,
\label{eq:eaCsB}\\
 \langle \lambda^*, v \rangle &= 0  \quad \text{$\forall v \in V : v = 0$ q.e. on $\{y^* = \Phi(y^*)\}$}\label{eq:eaGerd1},\\
 \langle \zeta^*, v \rangle &= 0  \quad  \text{$\forall v \in V : v = 0$ q.e. on $\{z^* = \Phi(z^*)\}$}\label{eq:eaGerd2},\\
\langle \xi_1^*, (p^*)^+ \rangle = \langle \xi_1^*, (p^*)^- \rangle = \langle \xi_2^*, (q^*)^+ \rangle = \langle \xi_2^*, (q^*)^- \rangle&= 0.\label{eq:eaXiP}
\end{align}
\end{subequations} 
\end{theorem}
The `q.e.' appearing in \eqref{eq:eaGerd1} and \eqref{eq:eaGerd2} means \emph{quasi-everywhere} and a statement holds q.e. if it holds everywhere except on a set of capacity zero. Let us observe that \eqref{eq:eaGerd1} and \eqref{eq:eaGerd2} imply
\begin{align*}
\langle \lambda^*, y^*-\Phi(y^*)\rangle &= 0\\
\langle \zeta^*, z^*-\Phi(z^*)\rangle &= 0.
\end{align*}
It is worth noting that if \cref{ass:forCStationarity} is not available, it is still possible to show that a subset of the conditions above (called \emph{weak C-stationarity}) are satisfied, see \cref{prop:weakCStationarity}.
Therein, \eqref{eq:eaGerd1}, \eqref{eq:eaGerd2} and \eqref{eq:eaXiP}
are missing.
By assuming just the continuity of
$(\cdot)^+ \colon V \to V$,
we can further show some substitutes for the missing relations,
see
\cref{prop:secondPart} and \cref{lem:complementarity_condition}.

\subsection{Examples}
\subsubsection{Obstacle map as solution map of PDE}\label{sec:example}
It is illustrative to give an example which occurs commonly in applications that satisfies all of the assumptions (for the subsolution and supersolution) that appear in this paper. Let $\Phi$ be increasing and satisfy $\Phi(0) \geq 0$ and let $F \in V^*_+$ be a given function. Define 
\begin{align*}
\underline u &:= 0\\
\overline u &:= A^{-1}F.
\end{align*}
We define the set of source terms
\[W:= \{ g \in V^* : 0 \leq g \leq F\}.\] 
With these choices, we in fact satisfy \textit{every} assumption on the existence of sub- and supersolutions that is mentioned in the paper. We will prove this later in \cref{lem:allSubSupSolnsSatisfied}.

Regarding specific choices of the function spaces, we can take $\Omega \subset \mathbb{R}^n$ to be a bounded Lipschitz domain and set $V=H^1_0(\Omega)$. A concrete example of $\Phi\colon H \to V$ could be: $\Phi(w) = \phi$ defined via
\begin{equation}
\begin{aligned}
-\Delta \phi &= w &&\text{in $\Omega$},\\
\phi&= 0&&\text{on $\partial\Omega$}\\
\end{aligned}
\end{equation}
where $-\Delta\colon V \to V^*$ denotes the weak Laplacian. Clearly, $\Phi(0) = 0$. Regarding the operator $A$, we could take for example $A(u) = -\grad \cdot (a\grad u)$ where the coefficient $a\colon \Omega \to \mathbb{R}$ is a function satisfying $a \in L^\infty(\Omega)$ and $a \geq a_0 > 0$ a.e.  for a constant $a_0$. 

Further applications and examples of QVIs can be found in e.g. \cite{AHROCQVI} and \cite{ChristofWachsmuth2021:1}. 

\subsubsection{An application in thermoforming}\label{sec:thermoforming}

We consider now an application of our theory in thermoforming. Thermoforming is a manufacturing process in which mould shapes that are to be reproduced are forced into contact with heated membranes (which are typically plastic sheets): the membrane deforms and takes on the shape of the mould. In some circumstances, the ensuing heat exchange between the materials leads also to a deformation of the mould, giving rise to the QVI nature of the problem. For further details, we refer to \cite{AHR}. We look at a concrete one-dimensional realisation from \cite[\S 4.3]{ACHP}, which is co-authored by two of the present authors.  Let $D \equiv \Omega = (0, 1)$, $A=-\Delta$, $V:=H^1_0(\Omega)$, and consider the QVI 
\begin{equation}
\label{eq:ModelQVINum}
\begin{aligned}
&\text{$u \in V$, } u \leq  \Phi(u),
\quad\langle -\Delta u - f, u-v \rangle_{V^*,V} \leq 0 \quad 
\forall v \in V, v \leq  \Phi(u),
\end{aligned}
\end{equation}
for $f \in L^2(\Omega)$ defined $f(x) = \pi^2\sin(\pi x)$, $\Phi(u)$ given by $\Phi(u) := \varphi T$ where $T \in H^1(\Omega)$ is the unique weak solution of
\begin{equation*}
\begin{aligned}
kT-\Delta T &= g(\psi T -u) &&\text{in $\Omega$},\\
\partial_\nu T &= 0 &&\text{on $\partial\Omega$,}
\end{aligned}
\end{equation*}
and where
\begin{equation*}
\begin{gathered}
\quad 
\varphi(x) =  
\frac{10 \pi^2\sin(\pi x)}{5 - \cos(2 \pi x)},
\quad 
k = \pi^2,
\quad
g(s)
=
4\min(0,s)^2,
\quad
 \psi(x) =   \frac{5 \pi^2\sin(\pi x)}{5 - \cos(2 \pi x)}.
\end{gathered}
\end{equation*}
Here, $u$ refers to the displacement of the membrane, $\Phi(u)$ is the displacement of the mould and $T$ is the temperature of the membrane. The above model is valid for one time step in the time discretisation of the thermoforming process, see again \cite{AHR} for full details.

Regarding existence for \eqref{eq:ModelQVINum}, first note that the fact that $g$ is decreasing implies that $u \mapsto T$ is increasing (the argument is the same as in \cite[Lemma 6.3]{AHR}) and hence as is the map $\Phi$.  Regarding $S(f,\cdot)$, since $f \in L^2(\Omega)$ with $f \geq 0$, it is easy to check that $0$ is a subsolution and $A^{-1}f$ is a supersolution   with $0 \leq A^{-1}f$ by non-negativity of $f$. 
By \cref{prop:existenceOfZ}, the QVI \eqref{eq:ModelQVINum} possesses minimal and maximal solutions on $[0, A^{-1}f]$. It is not difficult to see that $0$ is a solution of the QVI, and hence $\m(f) = 0$ for all non-negative $f \in L^2(\Omega)$. Note that \eqref{eq:ModelQVINum} has the second explicit solution $\sin(\pi x)$, see \cite[Lemma 4.3]{ACHP}.

The assumption \eqref{ass:PhiCC} on the complete continuity of $\Phi$ follows from the compact embedding of $V$ into $L^2(\Omega)$ and the inequality
\begin{align*}
\|\Phi(u_1) - \Phi(u_2)\|_{V} &\leq \Lip{g}
\left ( \|\varphi \|_{L^\infty(\Omega)}  k^{-1/2} 
+
\|  \varphi' \|_{L^\infty(\Omega)} k^{-1}
\right )\|u_1 - u_2\|_{L^2(\Omega)},
\end{align*}
which was shown in the proof of \cite[Theorem 3.9]{ACHP}.

The smallness condition \eqref{ass:PhiSmallLipschitzZ} on $B_{\epsilon^*}(0)$ is satisfied thanks to the next result.
\begin{lem}
\label{lem:ex2}
There exists an $\epsilon^* > 0$ such that $\Phi\colon B_{\epsilon^*}(0) \to V$ has a Lipschitz constant $C_L$ satisfying $C_L < 1$.
\end{lem}  
\begin{proof}

Let $B \subset V$ be a closed ball such that $B \subset \{v \in V \mid \norm{v}{V} < R^* \}$
where
\[R^*  :=
\frac{3}{10(1+\pi)}
\left (\sqrt{\frac{13\pi^2 + 8\pi}{80}} - \frac{\pi}{4}\right ).\] 
If we take $R <  R^*$ and set
\begin{equation}
\label{eq:MRformula}
M_R = \frac{1}{2}R + \frac{10(1+\pi)}{3\pi}R^2
\end{equation}
then we have,  using the notation $\mathrm{Lip}(g, I)$ to mean the Lipschitz constant of $g\colon I \to \mathbb{R}$ on the interval $I$, 
\begin{equation}
\label{eq:loc_lip_estimate}
\begin{aligned}
\mathrm{Lip}(g,[-M_R , M_R])
\left ( \|\varphi \|_{L^\infty(\Omega)}  k^{-1/2} 
+
\| \varphi' \|_{L^\infty(\Omega)} k^{-1}
\right )\frac{1}{\pi}
&= \frac{50}{3} \left (
R + \frac{20(1+\pi)}{3\pi}R^2
\right ) < 1
\end{aligned}
\end{equation}
(see \cite[Lemma 4.3]{ACHP} for the equality; the inequality holds because we took $R < R^*$ and is not difficult to see by using the quadratic formula). From the estimate \cite[Theorem 3.12]{ACHP}
\begin{align*}
	&\|\Phi(u_1) - \Phi(u_2)\|_{V}
        \\	
 &\quad\leq
	\frac{1}{\pi}\Lip{g, [-M_R , M_R]} 
\left ( \|\varphi \|_{L^\infty(\Omega)}  k^{-1/2} 
+
\| \varphi' \|_{L^\infty(\Omega)} k^{-1}
\right )\|u_1 - u_2\|_{V}
\quad \forall u_1, u_2 \in B_R(0)
\end{align*}
and \eqref{eq:loc_lip_estimate}, it follows that there exists $\gamma_B \in [0,1)$ such that 
\begin{equation}
\label{eq:ex:2:contra}
\| \Phi(u_1) - \Phi(u_2)\|_{V}
\leq
\gamma_B\| u_1 - u_2 \|_{V}
\quad \forall u_1, u_2 \in B.
\end{equation}
This implies the claim.
\end{proof}
Let us now address \eqref{ass:PhiDiffAtUNew}. First we remark that $\Phi$ is Newton differentiable from $V$ into $V \cap H^2(\Omega)$, see \cite[Theorem 3.9 (iii)]{ACHP}. Indeed, defining $\xi$ by
\begin{equation*}
\begin{aligned}
k\xi-\Delta \xi - g'(\psi T - u)\psi \xi &= -g'(\psi T -u)h &&\text{in $\Omega$},\\
\partial_\nu T &= 0 &&\text{on $\partial\Omega$,}
\end{aligned}
\end{equation*}
we have that the Newton derivative of $\Phi$ is $G\Phi(u)(h) = \varphi \xi$. 
In fact, we can prove the following stronger result.
\begin{lem}
The map $\Phi\colon V \to V$ is continuously Fr\'echet differentiable.
\end{lem}
\begin{proof}
This relies on applying the implicit function theorem to the map $\mathcal{F} \colon V \times H^1(\Omega) \to H^1(\Omega)^*$ defined by $\mathcal{F}(u,T) := kT-\Delta T-g(\psi T -u)$, and is essentially the same as the proof of \cite[Theorem 8]{AHR}, except with two differences. We modify the first step of the cited proof and show the Fr\'echet differentiability of $g$ as follows. Using the mean value theorem, for $x,y \in \mathbb{R}$, 
\begin{align*}
g(x+y)-g(x)-g'(x)y &=\int_0^1 g'(x+(1-\lambda)y)y - g'(x)y\;\mathrm{d}\lambda,
\end{align*}
and hence if we take now $v, d \in H^1(\Omega)$, using the Lipschitzness of $g'$,
\begin{align*}
\norm{g(v+d)-g(v)-g'(v)d}{H^1(\Omega)^*} &= \sup_{\substack{w \in H^1(\Omega)\\ \norm{w}{H^1(\Omega)} = 1}}\int_\Omega \left(\int_0^1 g'(v+(1-\lambda)d)d - g'(v)d\;\mathrm{d}\lambda\right) w \\
&\leq 8\sup_{\substack{w \in H^1(\Omega)\\ \norm{w}{H^1(\Omega)} = 1}}\int_\Omega \left(\int_0^1 (1-\lambda)d^2 \;\mathrm{d}\lambda\right) w \\
&\leq 8\sup_{\substack{w \in H^1(\Omega)\\ \norm{w}{H^1(\Omega)} = 1}}\int_\Omega d^2w \\
&\leq 8\norm{d}{L^4(\Omega)}^2\sup_{\substack{w \in H^1(\Omega)\\ \norm{w}{H^1(\Omega)} = 1}}\norm{w}{L^2(\Omega)}\\
&\leq C_1\norm{d}{H^1(\Omega)}^2\tag{using $H^1(\Omega) \cts L^4(\Omega)$}\
\end{align*}
for some constant $C_1$; this shows that $g\colon H^1(\Omega) \to H^1(\Omega)^*$ is Fr\'echet differentiable. In the second step of the proof of \cite[Theorem 8]{AHR}, we can use the Lipschitzness of $g'$ (instead of the mean value theorem as utilised there) to show the continuity of $g'\colon H^1(\Omega) \to \mathcal{L}(H^1(\Omega), H^1(\Omega)^*)$. The rest of the proof follows as in  \cite[Theorem 8]{AHR}.
\end{proof}

Now, if we take $W:=[0, F]$ where $F \in V^*$ is such that $F \geq f$, and set $\underline u := 0$ and $\overline u := S(F,\infty)=T_{\rho}(F,\infty)$ as the solution of the unconstrained problem, we see that \cref{ass:subAndSupersolutionsForZRho} is satisfied. Therefore, all of the results in \cref{sec:mainresults} up to and including \cref{sec:main_results_penalised} are applicable.

\section{Properties of the penalised problem}\label{sec:penalisedProblem}
This section culminates in a result that shows the existence (in a constructive way) of extremal solutions to \eqref{eq:penalisedPDEGeneralMRho}. To arrive at such a result, we first have to study some intermediary problems which will also be of considerable use in later sections.

Recalling $\sigma_\rho$ from \eqref{eq:mrhoHK}, let us point out that $\sigma_\rho\colon V \to V^*$ is bounded (in the sense of nonlinear operators), increasing, T-monotone and hemicontinuous\footnote{T-monotonicity in the nonlinear setting means $\langle \sigma_\rho(u)-\sigma_\rho(v), (u-v)^+\rangle \geq 0$ and hemicontinuity means $s \mapsto \langle \sigma_\rho(u+sv), w \rangle$ is continuous for all $u, v, w \in V$.}.   
T-monotonicity and the fact that $\sigma_\rho$ is increasing will be needed for the comparison results that are required for this paper.
Note that the T-monotonicity condition implies monotonicity \cite[Lemma 2.1, Chapter 2]{Mosco}. Another important property is the following, which shows that $\sigma_\rho$ is indeed a penalty operator.
\begin{lem}\label{lem:penaltyTypeCondition}
We have that
\[z_\rho \weaklyto z \text{ in $V$ and } \sigma_\rho(z_\rho) \to 0 \text{ in $V^*$} \implies z \leq 0.\]
\end{lem}
\begin{proof}
First observe that for any $h \in H$, we have $\sigma_\rho(h) \to h^+$ in $H$.
This is an immediate consequence of the estimate
\[0 \leq r^+ - \sigma_\rho(r) \leq \frac{\rho}{2}\]
(see \cite[Lemma 2.1 (iv)]{MR2822818}). Suppose that $z_\rho \weaklyto z$ in $V$ and $\sigma_\rho(z_\rho) \to 0$ in $V^*$. By monotonicity, we have for any $\lambda > 0$,
\[\langle \sigma_\rho(z_\rho)-\sigma_\rho(z+\lambda v), z_\rho - z - \lambda v \rangle \geq 0 \quad \forall v \in V.\]
Passing to the limit $\rho \searrow 0$ here using the strong convergence of $\sigma_\rho(z_\rho)$  and the fact that $\sigma_\rho(z+\lambda v) \to (z+\lambda v)^+$ in $H$, we obtain
\[\langle (z+\lambda v)^+, \lambda v \rangle \geq 0 \quad \forall v \in V.\]
Dividing through by $\lambda$ and using (hemi-)continuity of $(\cdot)^+\colon H \to H$, we derive
\[\langle z^+, v \rangle \geq 0 \quad \forall v \in V.\]
The arbitrariness of $v$ then implies that $z^+=0$.
\end{proof}
\subsection{Results on a semilinear elliptic PDE}\label{sec:resultsOnSemilinear}

For $f \in V^*$ and $\varphi \in H$, consider the equation
\begin{equation}\label{eq:ppFixed}
Au + \frac{1}{\rho}\sigma_\rho(u-\Phi(\varphi)) = f,
\end{equation}
the solution map of which we write
\[u  = T_\rho(f,\varphi),\]
so that $T_\rho\colon V^* \times H \to V$.
The equation \eqref{eq:ppFixed} has a unique solution (for fixed $f$ and $\varphi$): the nonlinearity is monotone, radially continuous and bounded, giving pseudomonotonicity of the full elliptic operator by \cite[Lemma 2.9 and Lemma 2.11]{Roubicek} whereas coercivity follows from
\[\langle Au, u-\Phi(\varphi) \rangle + \frac{1}{\rho}\langle \sigma_\rho(u-\Phi(\varphi)), u-\Phi(\varphi)) \geq C_a\norm{u}{V}^2 - C_b\norm{u}{V}\norm{\Phi(\varphi)}{V},\]
leading to existence via \cite[Theorem 2.6]{Roubicek}.

In the next two lemmas,
we utilise the results of \cite{Wachsmuth2021:2}
to obtain Lipschitz estimates
for  $T_\rho$.
\begin{lem}
	\label{lem:small_estimate}
Assume that $\Phi$ is Lipschitz on $U \subset V$ with Lipschitz constant $C_L \ge 0$ satisfying
	\begin{equation}\label{ass:generalLipschitzConstant}
								\text{$C_L < \frac{C_a}{C_b}$}\qquad\text{or}\qquad\text{$A$ is self-adjoint and $C_L < 2\frac{ \sqrt{C_b/C_a}}{1 + C_b/C_a}$.}
\end{equation}
	Then,
	there exist constants $C \ge 0$, $\tilde c \in [0,1)$
	(depending only on $C_L$, $C_a$, $C_b$ and the self-adjointness of $A$)
	such that
	for all
	$u,v \in V$ and $\varphi, \psi \in U$,
	we have
	\begin{equation*}
		\dual{A (u - v)}{u - \Phi(\varphi) - v + \Phi(\psi)}
		\ge
		C \parens[\big]{
			\norm{u - v}V^2 - \tilde c^2 \norm{\varphi - \psi}V^2
		}
		.
	\end{equation*}
\end{lem}
\begin{proof}
	This is precisely \cite[Lemma~20]{Wachsmuth2021:2}.
	Note that the linear and continuous operator $A$ is a derivative of a convex function
	if and only if $A$ is self-adjoint.
\end{proof}
If the constant $C_L$ is larger than or equal to the allowed threshold
from \cref{lem:small_estimate},
the result no longer holds,
cf.\ \cite[Theorems~3.6, 3.7]{Wachsmuth2019:2}. 
Note that the latter constant in \eqref{ass:generalLipschitzConstant} is larger than the former one. If $C_L < C_a\slash C_b$, we may choose
\[C= \frac{C_a}{2}, \quad \tilde c = \frac{C_bC_L}{C_a}\]
whereas in the other case we could choose
\[C= \frac{C_aC_b}{C_a + C_b}, \quad  \tilde c = \frac{(C_a+C_b)C_L}{2\sqrt{C_aC_b}}.\]
The next result will be crucial, since it shows that the map
$u \mapsto T_\rho(f, u)$ is a contraction
under appropriate assumptions.

\begin{prop}\label{lem:TrhoLipschitz}
	For all $f, g \in V^*$ and $\varphi, \psi \in H$, we have	
	\[
		\norm{T_\rho(f,\varphi) - T_\rho(g,\psi)}{V}
		\leq
		\sqrt2 C_a^{-1}\norm{f-g}{V^*} + C_a^{-1}(\sqrt{2} C_b)\norm{\Phi(\psi)-\Phi(\varphi)}{V}.
	\]
	In case that
	$\Phi\colon V \to V$ is locally Lipschitz in $U \subset V$
	with small Lipschitz constant $C_L$ satisfying \eqref{ass:generalLipschitzConstant} 
	and if
	$\varphi,\psi  \in U$, then
	\[
		\norm{T_\rho(f,\varphi) - T_\rho(g,\psi)}{V}
		\leq
		\hat C \norm{f-g}{V^*} + \hat c \norm{\psi-\varphi}{V}
	\]
	for some constants $\hat C \ge 0$, $\hat c \in [0,1)$,
	depending only on $C_L$, $C_a$, $C_b$ and the self-adjointness of $A$.
\end{prop}
\begin{proof}
	Let $u=T_\rho(f,\varphi)$ and $v=T_\rho(g,\psi)$.
	We have that
	\begin{align*}
		A u + \frac 1\rho \sigma_\rho(u - \Phi(\varphi)) = f \qquad\text{and}\qquad 
		A v + \frac 1\rho \sigma_\rho(v - \Phi(\psi)) &= g	.
	\end{align*}
	Testing the difference with $u - \Phi(\varphi) - v + \Phi(\psi)$
	and using monotonicity
	leads to
	\begin{equation}
		\label{eq:tested_VIs}
		\dual{A (u - v)}{u - \Phi(\varphi) - v + \Phi(\psi)}
		\le
		\dual{f - g}{u - \Phi(\varphi) - v + \Phi(\psi)}
	\end{equation}
	and, consequently,
	\begin{equation*}
		C_a \norm{u - v}V^2 - C_b \norm{u - v}V \norm{\Phi(\varphi) - \Phi(\psi)}V
		\le
		\norm{f - g}{V^*}\parens[\big]{\norm{u - v}V + \norm{\Phi(\varphi) - \Phi(\psi)}V}.
	\end{equation*}
	Together with the estimates
	\begin{align*}
		C_b \norm{u - v}V \norm{\Phi(\varphi) - \Phi(\psi)}V
		&\le
		\frac{C_a}{4}\norm{u - v}V^2 + \frac{C_b^2}{C_a}\norm{\Phi(\varphi) - \Phi(\psi)}V^2
		,
		\\
		\norm{f - g}{V^*}\norm{u - v}V
		&\le
		\frac{C_a}{4}\norm{u - v}V^2 + \frac{1}{C_a}\norm{f - g}{V^*}^2,
	\end{align*}
	we get
	\begin{align*}
		\frac{1}{2} \norm{u - v}V^2
		&\le
		\frac{C_b^2}{C_a^2}\norm{\Phi(\varphi) - \Phi(\psi)}V^2
		+
		\frac{1}{C_a}\norm{f - g}{V^*}\norm{\Phi(\varphi) - \Phi(\psi)}V
		+
		\frac{1}{C_a^2}\norm{f - g}{V^*}^2
		\\
		&\le
		\parens*{
			\frac{C_b}{C_a}\norm{\Phi(\varphi) - \Phi(\psi)}V
			+
			\frac{1}{C_a}\norm{f - g}{V^*}
		}^2.
	\end{align*}
	This shows the first estimate.

	In order to arrive at the second estimate, we use
	\cref{lem:small_estimate}
	in \eqref{eq:tested_VIs}
	to obtain
	\begin{equation*}
		C \parens[\big]{
			\norm{u - v}V^2 - \tilde c^2 \norm{\varphi - \psi}V^2
		}
		\le
		\norm{f - g}{V^*}
		\parens*{\norm{u - v}V + C_L \norm{\varphi - \psi}V }
		.
	\end{equation*}
	Together with
	\begin{equation*}
		\norm{f - g}{V^*}
		\norm{u - v}V
		\le
		C \frac{1 - \tilde c^2}{2} \norm{u - v}V^2
		+
		\frac1{2C (1 - \tilde c^2)} \norm{f - g}{V^*}^2
	\end{equation*}
	we get
	\begin{align*}
		C \frac{1+\tilde c^2}{2} \norm{u - v}V^2
		&\le
		\frac{1}{2 C (1-\tilde c^2)}\norm{f - g}{V^*}^2
		+
		\norm{f - g}{V^*}
		C_L \norm{\varphi - \psi}V
		+
		C\tilde c^2 \norm{\varphi - \psi}V^2
		\\
		&\le
		\parens*{
			\parens*{
				\frac{1}{\sqrt{2 C (1-\tilde c^2)}}
				+
				\frac{C_L}{2 \sqrt{C} \tilde c}
			}\norm{f - g}{V^*}
			+
			\sqrt{C}\tilde c \norm{\varphi - \psi}V
		}^2
		.
	\end{align*}
	This yields the claim.
\end{proof}
The next lemma shows that the solution of the PDE converges to the solution of associated VI.
\begin{lem}\label{lem:PDEApproximatesVI}
For $f \in V^*$ and $\varphi \in H$, we have $T_\rho(f,\varphi) \to S(f,\varphi)$ in $V$ as $\rho \searrow 0$.
\end{lem}
\begin{proof}
This is an extension of the classical penalty theory (see \cite[Theorem 3.1]{Glowinski} or \cite[\S 5.3, Chapter 3]{Lions1969}) to the varying $\sigma_\rho$ setting given in \cite{AHROCQVI}. More precisely, since
$\sigma_\rho$ is hemicontinuous (hence radially continuous) and bounded, this follows by \cite[Theorem 2.18]{AHROCQVI}. 
\end{proof}

\subsection{Order properties}\label{sec:orderProperties}
In this section, we discuss various properties related to the partial order. The next lemma is fundamental: it will be used to show that \eqref{eq:penalisedPDEGeneralMRho} has minimal and maximal solutions.

\begin{lem}\label{lem:TRhoIncreasingMap}The map $T_\rho(\cdot,\cdot)\colon V^* \times H \to V$ is increasing.
\end{lem}
\begin{proof}
Let $f \geq g$, $\varphi \geq \psi$ and consider $u=T_\rho(f,\varphi)$ and $v=T_\rho(g,\psi)$.  Testing the equation for $v-u$ with $(v-u)^+$, we have
\[\langle A(v-u), (v-u)^+ \rangle + \frac 1 \rho \langle \sigma_\rho(v-\Phi(\psi)) - \sigma_\rho(u-\Phi(\varphi)), (v-u)^+\rangle = \langle g-f, (v-u)^+ \rangle.\]
Since $\Phi(\varphi) \geq \Phi(\psi)$ we have $v-\Phi(\psi) \geq v-\Phi(\varphi)$ and hence  by the increasingness property, $\sigma_\rho(v-\Phi(\psi)) \geq \sigma_\rho(v-\Phi(\varphi))$.  This implies from above that
\[\langle A(v-u), (v-u)^+ \rangle + \frac 1 \rho \langle \sigma_\rho(v-\Phi(\varphi)) - \sigma_\rho(u-\Phi(\varphi)), (v-u)^+\rangle \leq 0\]
and hence, using T-monotonicity, we get $(v-u)^+=0$ so that $v \leq u$.
\end{proof}

\begin{lem}\label{lem:TRhoIncreasingInRho}
We have
\[\rho \leq \kappa \implies T_\rho(f,\varphi) \leq T_\kappa (f,\varphi).\]
\end{lem}
\begin{proof}
Let $u_\rho=T_\rho(f,\varphi)$ and $u_\kappa=T_\kappa(f,\varphi)$. We have
\[A(u_\rho - u_\kappa) + \frac 1\rho \sigma_\rho(u_\rho-\Phi(\varphi)) - \frac 1\kappa \sigma_\kappa(u_\kappa - \Phi(\varphi)) = 0,\]
and we manipulate
\begin{align*}
&\frac 1\rho \sigma_\rho(u_\rho-\Phi(\varphi)) - \frac 1\kappa \sigma_\kappa(u_\kappa - \Phi(\varphi))\\ &= \left(\frac 1\rho-\frac 1\kappa\right) \sigma_\rho(u_\rho-\Phi(\varphi)) + \frac 1\kappa \left(\sigma_\rho(u_\rho-\Phi(\varphi))-\sigma_\kappa(u_\kappa - \Phi(\varphi))\right)\\
&= \left(\frac 1\rho-\frac 1\kappa\right) \sigma_\rho(u_\rho-\Phi(\varphi)) + \frac 1\kappa \left(\sigma_\rho(u_\rho-\Phi(\varphi))-\sigma_\rho(u_\kappa - \Phi(\varphi))\right) + \frac 1\kappa \left(\sigma_\rho(u_\kappa-\Phi(\varphi)) -\sigma_\kappa(u_\kappa - \Phi(\varphi))\right)  
\end{align*}
which, when tested with $(u_\rho-u_\kappa)^+$, is non-negative (the first term by $\rho \leq \kappa$, the second by T-monotonicity and the third because $\sigma_\rho$ satisfies $\rho \leq \kappa \implies \sigma_\rho \geq \sigma_\kappa$).
\end{proof}
We should expect that the solution of the VI is dominated by the solution of the penalised equation.

\begin{lem}\label{lem:VIisLessThanPDE}
We have $S(f,\varphi) \leq T_\rho(f,\varphi)$.
\end{lem}
\begin{proof}
 Let $u_\rho =T_\rho(f,\varphi)$ and $v=S(f,\varphi)$. Take as test function in the VI for $v$ the function $v-(v-u_\rho)^+$ and combine to get
\[\langle A(v-u_\rho), (v-u_\rho)^+\rangle - \frac 1\rho \langle \sigma_\rho(u_\rho-\Phi(\varphi)), (v-u_\rho)^+ \rangle \leq 0.\]
Since $v \leq \Phi(\varphi)$, we have $u_\rho-\Phi(\varphi) \leq u_\rho -v$ and the increasing property of $\sigma_\rho$ as well as the fact that $\sigma_\rho \equiv 0$ on $(-\infty,0]$ implies that
\[\langle \sigma_\rho(u_\rho-\Phi(\varphi)), (v-u_\rho)^+ \rangle \leq \langle \sigma_\rho(u_\rho-v), (v-u_\rho)^+ \rangle \leq 0.\]
Using this fact above, we deduce that $\langle A(v-u_\rho), (v-u_\rho)^+\rangle \leq 0$ which gives the claim.
\end{proof}
Before we move on, let us prove, with the aid of a result from this section, a claim we made earlier in \cref{sec:example}.
\begin{lem}\label{lem:allSubSupSolnsSatisfied}
The example in \cref{sec:example} satisfies every assumption on sub- and supersolutions in the paper. More precisely, with $W$, $\underline u$ and $\overline u$ defined as in \cref{sec:example},  any $f \in W$ and the set $W$ satisfy \cref{ass:forZSubAndSuperSolutions}, \cref{ass:subAndSupersolutionsZ}, \cref{ass:subAndSupersolutionsForZRho}, \cref{ass:ocSubAndSupersolutions} and \cref{ass:ocSubAndSupersolutionsStronger} (with $W=\Uad$), \cref{ass:forZRhoSubAndSuperSolutions} and \cref{ass:subAndSuperSolutions}.
\end{lem}
\begin{proof}
It suffices to show that $\underline u$ and $\overline u$ are sub- and supersolutions for $S(f,\cdot)$ and $T_\rho(f,\cdot)$ for all $f \in W$. It is not difficult to see this:
\begin{itemize}
\item Since $\Phi$ is increasing,  for all $\rho \geq 0$, we have $\overline u = T_\rho(F, \infty) \geq T_\rho(F, \overline u) \geq T_\rho(f, \overline u)$ for any $f \leq F$ because of \cref{lem:TRhoIncreasingMap} (for $\rho >0)$ and \cite[\S 4:5, Theorem 5.1]{Rodrigues} (for $\rho=0$). Thus $\overline u$ is a supersolution of $S(f, \cdot)$ and $T_\rho(f, \cdot)$ for all $f \in W$.

\item If $f \in W$, for all $\rho \geq 0$, we have $T_\rho(f,0) \geq T_\rho(0,0) = 0$ again by the above-cited results and since $f \geq 0$. Hence $\underline u$ is a subsolution for $S(f,\cdot)$ and $T_\rho(f,\cdot)$ for all $f \in W$.

\end{itemize}

\end{proof}

\subsection{Minimal and maximal solutions of PDEs}\label{sec:minMaxSolnsOfPDEs}
Recall \eqref{eq:penalisedPDEGeneralMRho}:
\[Au+ \frac{1}{\rho}\sigma_\rho(u-\Phi(u)) = f.\]
Let us assume the existence of a sub- and supersolution for $T_\rho(f,\cdot)$ and prove our earlier claim that \eqref{eq:penalisedPDEGeneralMRho} has extremal solutions.
\begin{ass}[Well definedness of $\Z_\rho(f)$]\label{ass:forZRhoSubAndSuperSolutions}
Given $f \in V^*$, assume that there exist $\underline u, \overline u \in V$ such that
\[\underline u \leq T_\rho(f, \underline u), \quad \overline{u} \geq T_\rho(f, \overline u), \quad\text{and}\quad  \underline u \leq \overline u.\]
\end{ass}
This assumption is exactly \eqref{ass:newTRhoSubSup}. Arguing like in \cref{prop:existenceOfZ}, we have the following.
\begin{prop}\label{prop:existenceOfZRho}Under \cref{ass:forZRhoSubAndSuperSolutions}, there exist a minimal solution $\m_\rho(f)$ and maximal solution $\M_\rho(f)$ to the equation \eqref{eq:penalisedPDEGeneralMRho}
on $[\underline u, \overline u]$.
\end{prop}
\begin{proof}
Due to \cref{lem:TRhoIncreasingMap}, it follows by the Birkhoff--Tartar theorem \cite[\S 15.2, Proposition 2]{MR556865} that the set of fixed points of $u \mapsto T_\rho(f,u)$ is non-empty and possesses a minimal and maximal solution on the interval $[\underline u, \overline u].$
\end{proof}
Now, we focus on ways to approximate these extremal solutions by sequences.
\begin{defn}\label{defn:iterativeSeq}Define the iterative sequence $\{\overline u^n_\rho\}$ by 
\begin{align*}
\overline u^n_\rho &= T_\rho(f, \overline u^{n-1}_\rho),\\
\overline u^0_\rho &= \overline u,
\end{align*}
and $\{\underline u^n_\rho\}$ by
\begin{align*}
\underline u^n_\rho &= T_\rho(f, \underline u^{n-1}_\rho),\\
\underline u^0_\rho &= \underline u.
\end{align*}
\end{defn}
Note that $\{\overline u^n_\rho\}$ is a decreasing sequence and $\{\underline u^n_\rho\}$ is an increasing sequence (see the proof of the next result). In fact, $\overline u^n_\rho$ approaches $\M_\rho(f)$ from above and $\underline u^n_\rho$ approaches $\m_\rho(f)$ from below.

\begin{prop}[Strong convergence]\label{prop:strongConvergenceOfPenVIProblems}
Under
\cref{ass:forZRhoSubAndSuperSolutions}, assume \eqref{ass:PhiCC} or that
\begin{align}
&\Phi\colon V \to V \text{ is weakly sequentially continuous}\label{ass:PhiWeakSeqCts},\\
 &V \ctsCompact H\label{ass:VCompactInH}.
 \end{align}
Then
\[\text{$\overline u^n_\rho \searrow \M_\rho(f)$ and $\underline u^n_\rho \nearrow \m_\rho(f)$ strongly in $V$ as $n \to \infty$.}\]
\end{prop}
\begin{proof}
For readibility, let us write $u^n$ instead of $\overline u^n_\rho$. Each $u^n$ satisfies
\[Au^n + \frac{1}{\rho}\sigma_\rho(u^n-\Phi(u^{n-1})) = f.\]
By definition of supersolution, $u^0 = \overline u \geq T_\rho(f, \overline u) = u^1$, and since we have shown above that $T_\rho(f,\cdot)$ is increasing, we obtain in this fashion that $u^n \geq u^{n+1}$ so that $\{u^n\}$ is a decreasing sequence. 

Note also that $u^1 \geq T_\rho(f,\underline{u}) \geq \underline{u}$, hence $u^n \geq \underline{u}$ for all $n$. Define $v_0 = \Phi(\underline u)$. Then we have $v_0  \leq \Phi(u^n)$ for all $n$ since $\Phi$ is increasing, therefore,
\begin{align*}
\langle \sigma_\rho(u^n-\Phi(u^{n-1})), u^n-v_0 \rangle = \langle \sigma_\rho(u^n-\Phi(u^{n-1})) - \sigma_\rho(v_0-\Phi(u^{n-1})), u^n-v_0\rangle \geq 0
\end{align*}
by monotonicity. Testing the $u^n$ equation with $u^n-v_0$,
\[C_a\norm{u^n}{V}^2 \leq \norm{f}{V^*}\norm{u^n}{V} + \norm{f}{V^*}\norm{v_0}{V}+ C_b\norm{u^n}{V}\norm{v_0}{V}\]
and this leads to a uniform bound in $V$. 
Thus $u^n \weaklyto u$ in $V$
for some $u$, for the entire sequence by monotonicity (see e.g. \cite[Lemma 2.3]{AHROCQVI}).  Take any solution $u^* = T_\rho(f,u^*)$ with $u^* \leq u^0$. It follows that $u^* \leq u^1$ by applying $T_\rho(f,\cdot)$ to both sides. Likewise, $u^* \leq u^n$ and hence $u^* \leq u$, so if $u$ is a solution of the limiting problem, it must be the largest solution. Let us show now that $u$ does solve the limiting equation, i.e., that $u=T_\rho(f,u)$.

\medskip

\noindent \emph{Satisfaction of the equation.}

\noindent \emph{Case 1.} Under complete continuity \eqref{ass:PhiCC}, making the transformation $w^n = u^n - \Phi(u^{n-1})$, we can write the equation for $u^n$ as
\[Aw^n + \frac 1\rho \sigma_\rho(w^n) = f - A\Phi(u^{n-1}).\]
Call the operator on the left-hand side $\hat A$. By monotonicity, we have for all $v \in V$,
\[0 \leq \langle \hat A(w^n) - \hat A(v), w^n -v \rangle  = \langle f- A\Phi(u^{n-1}) - \hat A(v), w^n -v \rangle,\]
and hence, noting that $w^n \weaklyto u-\Phi(u) =: w$ and $\Phi(u^{n-1}) \to \Phi(u)$ by \eqref{ass:PhiCC}  (observe that it suffices to have this complete continuity only for monotonic sequences),
\[0 \leq \langle f- A\Phi(u) - \hat A(v), w - v\rangle \quad \forall v \in V.\]
Since $\hat A$ is radially continuous, by Minty's trick \cite[Lemma 2.13]{Roubicek}, we obtain $\hat A(w) = f- A\Phi(u)$, i.e.,
\[Aw + \frac 1\rho \sigma_\rho(w ) = f - A\Phi(u).\]
Since $w = u-\Phi(u)$,  we see that $u = T_\rho(f,u)$.

\medskip

\noindent \emph{Case 2.}  Otherwise, by \eqref{ass:PhiWeakSeqCts}
and the Lipschitz continuity of $\sigma_\rho\colon H \to H$ and the fact that $V \ctsCompact H$, we obtain $\sigma_\rho(u^n-\Phi(u^{n-1})) \weaklyto \sigma_\rho(u-\Phi(u))$ in $V^*$. This lets us pass to the limit in the equation for $u^n$.

\medskip

\noindent \emph{Strong convergence.}
It remains for us to show that $u_n \to u$ in $V$ strongly.

\noindent \emph{Case 1.} By using \cref{lem:TrhoLipschitz}, we obtain the continuous dependence estimate
\[\norm{u^n-u}{V} \leq  C\norm{\Phi(u^{n-1})-\Phi(u)}{V},\]
we can pass to the limit on the right-hand side using \eqref{ass:PhiCC},  yielding $u^n \to u$.

\medskip

\noindent \emph{Case 2.} In the second case, we test the equation for $u^n-u$ with $u^n-u$ and manipulate
\begin{align*}
    C_a\rho\norm{u^n-u}{V}^2 &\leq \langle \sigma_\rho(u-\Phi(u))-\sigma_\rho(u^n-\Phi(u^{n-1})) , u^n-u \rangle_{H^*,H}\\
 & \to 0
\end{align*}
with the convergence because we have $\sigma_\rho(u^n-\Phi(u^{n-1})) \to \sigma_\rho(u-\Phi(u))$  in $H^*$ by the compact embedding \eqref{ass:VCompactInH}, and  $u^n-u \to 0$ in $H$ for the same reason.
\end{proof}
\begin{remark}
If we assume that $\Phi \colon H \to V$ is continuous, \eqref{ass:VCompactInH} implies \eqref{ass:PhiCC}. Since the aforementioned continuity of $\Phi$ and \eqref{ass:VCompactInH} typically do hold in examples, the above result is rather a powerful property that we attain without cost.
\end{remark}
In some sense, the conclusion of \cref{prop:strongConvergenceOfPenVIProblems} improves the similar convergence result of \cite[Theorem 2.18]{AHROCQVI} where it was shown that, in greater generality and in the absence of the assumption that $\Phi$ is increasing, solutions of \eqref{eq:penalisedPDEGeneralMRho} converge along a subsequence to some solution of the QVI \eqref{eq:QVIIntro}. Here, we are able to select precisely the minimal or maximal solution as the limiting objects thanks to the strengthened structure.

\section{Convergence to the QVIs}\label{sec:convergencetoQVIs}
We now consider the limiting behaviour of $\M_\rho$ and $\m_\rho$ as $\rho \searrow 0$ and show that they converge to the expected limits under some circumstance. First, we need some more properties.
\subsection{Properties with respect to varying $\rho$}
In the next lemma, we show that $\rho \mapsto \Z_\rho$ is increasing. In other words, $\Z_\rho$ shrinks as $\rho$ gets smaller (this is natural since we expect $\Z_\rho$ to converge to the solution of the constrained problem). Recall Definition \ref{defn:iterativeSeq}. 
\begin{lem}\label{lem:barURhoIsDecreasing}
Let $\rho, \kappa > 0$ and assume that $\underline u$ is a subsolution and $\overline u$ is a supersolution of  both $T_\rho(f,\cdot)$ and $T_\kappa(f,\cdot)$ with $\underline u \leq \overline u$. If $\rho \leq \kappa$, then 
\[\text{$\overline u^n_\rho \leq \overline u^n_\kappa\qquad$  and $\qquad\underline{u}^n_\rho \leq  \underline{u}^n_\kappa$}.\]
Thus if the assumptions of \cref{prop:strongConvergenceOfPenVIProblems} hold, then 
\[\text{$\M_\rho(f) \leq \M_\kappa(f)\qquad$ and $\qquad\m_\rho(f) \leq \m_\kappa(f)$}.\]
\end{lem}
\begin{proof}
Set $\overline u_\rho:=\M_\rho(f)$ $\overline u_\kappa:= \M_\kappa(f)$. We have $\overline u^1_\rho = T_\rho(f,\overline u) \leq T_\kappa(f, \overline u) =\overline u^1_\kappa$ by \cref{lem:TRhoIncreasingInRho}. Hence $\overline u^2_\rho = T_\rho(f, \overline u^1_\rho) \leq T_\rho(f, \overline u^1_\kappa) \leq T_\kappa(f, \overline u^1_\kappa) = \overline u^2_\kappa$ by the increasing property of \cref{lem:TRhoIncreasingMap} and again \cref{lem:TRhoIncreasingInRho}. The same holds when one replaces the supersolution by the subsolution. This implies the first claim and then taking $n \to \infty$, using \cref{prop:strongConvergenceOfPenVIProblems} implies the second.
\end{proof}

\begin{remark}In the above lemma, we could consider two different pairs of sub/supersolution for $T_\rho$ and $T_\kappa$. We can prove the same result (but $\Z_\rho$ and $\Z_\kappa$ would be defined on different intervals of course) if we assume the subsolution (supersolution) for the $\rho$ problem is less than or equal to than the subsolution (supersolution) for the $\kappa$ problem. We leave the details to the reader.
\end{remark}
In a similar fashion to \cref{defn:iterativeSeq}, we introduce the following.

\begin{defn}
Define the sequences
\begin{align*}
\hat u^n &= S(f, \hat u^{n-1}),\\
\hat u^0 &= \overline u,
\end{align*}
and
\begin{align*}
\tilde u^n &= S(f, \tilde u^{n-1}),\\
\tilde u^0 &= \underline u.
\end{align*}
\end{defn}
The map $S$ can be thought of as $T_0$ (i.e., $T_\rho$ with $\rho=0$).

\begin{prop}[Lemmas 3.2 and 3.3 of \cite{AHRExtremals}]\label{lem:convergenceOfVIsToQVIs}
Under \cref{ass:forZSubAndSuperSolutions}, assume \eqref{ass:PhiCC}.
Then $\hat u^n \searrow \M(f)$ and $\tilde {u}^n  \nearrow \m(f)$ in $V$.
\end{prop}
This proposition corresponds to the convergence results of \cref{prop:strongConvergenceOfPenVIProblems} for the $\rho=0$ case.

The next lemma shows that the unconstrained iterates (which solve PDEs) are greater than the constrained iterates which solve VIs).
\begin{lem}\label{lem:comparisonRhoAndVI}
If $\{\overline u^n_\rho\}, \{\hat u^n\}$ and $\{\underline u^n_\rho\}, \{\tilde u^n\}$ are defined as above with the same initial elements $\overline u$ and $\underline u$ respectively, then $\overline u^n_\rho \geq \hat u^n$  and $\underline u^n_\rho \geq \tilde u^n$.
\end{lem}
\begin{proof}
This is essentially \cref{lem:barURhoIsDecreasing} with $\rho=0$.

We have, using \cref{lem:VIisLessThanPDE} , $\overline u^1_\rho = T_\rho(f, \overline u) \geq S(f,\overline u) = \hat u_1$, and hence $\overline u^2_\rho = T_\rho(f,\overline u^1_\rho) \geq S(f, \overline u^1_\rho) \geq S(f,\hat u_1) = \hat u_2$, and so on. Here, we have used the increasing property of $S(f,\cdot)$. The same applies with the supersolution replacing the subsolution.
\end{proof}

\subsection{The $\rho \searrow 0$ limit}
We want to prove that the penalised extremal solutions converge to the (non-penalised) extremal solutions in the limit $\rho \searrow 0$. First, we have to guarantee that all these objects exist.
\begin{ass}[Well definedness of $\Z(f)$ and $\Z_\rho(f)$ for all $\rho$ sufficiently small]\label{ass:subAndSuperSolutions}
Given $f \in V^*$, assume that
 there exist $\underline u, \overline u \in V$ and $\rho_0 >0$ such that
\begin{equation*}
\begin{aligned}
\underline u &\leq \overline u,\\
\underline u &\leq S(f, \underline u),\\
\overline u &\geq T_{\rho_0}(f,\overline u). 
\end{aligned}
\end{equation*}
\end{ass}
\begin{remark}\label{rem:equivalenceOfSubAssumptions}
The statements
\begin{equation*}
\underline u \leq S(f,\underline u)
\end{equation*}
and
\begin{equation*}
\underline u \leq T_\rho(f,\underline u) \quad \forall \rho \leq \rho_0
\end{equation*}
are equivalent. One direction follows from the convergence result (as $\rho \to 0$) of \cref{lem:PDEApproximatesVI} and the other from \cref{lem:VIisLessThanPDE}.
\end{remark}

Under this assumption, $\underline u \leq S(f,\underline u) \leq T_\rho(f, \underline u)$ for all $\rho$ (see the above remark), and 
\[\overline u \geq T_{\rho_0}(f, \overline u) \geq T_\rho(f,\overline u) \geq S(f, \overline u) \quad \forall \rho \leq \rho_0\] so that $(\underline u, \overline u)$ are a sub- and supersolution pair for both $T_\rho(f, \cdot)$ (for all $\rho \leq \rho_0)$) and $S(f, \cdot)$. This means that both \cref{ass:forZSubAndSuperSolutions} and \eqref{ass:newTRhoSubSup} (i.e., \cref{ass:forZRhoSubAndSuperSolutions}) are satisfied and both $\Z_\rho(f)$ and $\Z(f)$ are well defined objects in $[\underline u, \overline u]$, for all $\rho \leq \rho_0$.

\begin{theorem}\label{thm:convergenceOfMRho}
Assume
\cref{ass:subAndSuperSolutions}
and the weak sequential continuity \eqref{ass:PhiWeakSeqCts}.
Then $\M_\rho(f) \searrow \M(f)$ weakly in $V$ and $\m_\rho(f) \searrow u$ weakly in $V$,
where $u \in V$ is a solution of \eqref{eq:QVIIntro}.
If \eqref{ass:PhiCC} holds, then the convergences are strong. 
\end{theorem}
\begin{proof}
As mentioned above, $\M(f)$ and $\M_\rho(f)$ exist for all $0 < \rho \leq \rho_0$ by \cref{ass:subAndSuperSolutions}. Define $u_\rho := \M_\rho(f)$.  Since $\underline u \leq S(f,\underline u) \leq \Phi(\underline u) \leq \Phi(u_\rho)$ for all $\rho$, if we set $v_0 := \underline u$, we can test the $u_\rho$ equation with $u_\rho -v_0$ and we get that $\{u_\rho\}$ is bounded by using
\begin{align*}
\langle \sigma_\rho(u_\rho-\Phi(u_\rho)), u_\rho - v_0 \rangle = 
\langle \sigma_\rho(u_\rho-\Phi(u_\rho)) - \sigma_\rho(v_0 - \Phi(u_\rho)), u_\rho - v_0 \rangle  \geq 0.
\end{align*}
Hence $u_\rho \weaklyto u$ in $V$ for a subsequence that we have relabelled. This implies that $\rho (f - Au_\rho) = \sigma_\rho(u_\rho -\Phi(u_\rho)) \to 0$ in $V^*$. Then, testing the equation for $u_\rho$ with $u_\rho-v$ for $v \in V$ and using the monotonicity formula 
\begin{align}
\langle \sigma_\rho(u_\rho - \Phi(u_\rho)), u_\rho-v\rangle &\geq \langle \sigma_\rho(v-\Phi(u_\rho)), u_\rho-v\rangle\quad \forall v \in V\label{eq:monotonicityIneq}
\end{align}
(this follows by the monotonicity of $\sigma_\rho$; see the proof of \cite[Theorem 2.18]{AHROCQVI}),
we have
\[\langle Au_\rho, u_\rho \rangle + \frac 1\rho \langle \sigma_\rho(v-\Phi(u_\rho)), u_\rho - v \rangle \leq \langle f, u_\rho - v \rangle + \langle Au_\rho, v \rangle.\]
Let $v \in V$ be such that $v \leq \Phi(u)$. Since $u \leq u_\rho $ (as \cref{lem:barURhoIsDecreasing} shows that $\{u_\rho\}$ is a decreasing sequence), $v \leq \Phi(u_\rho)$ and hence the second term on the left disappears. We can pass to the limit and use weak lower semicontinuity to obtain that $u$ solves the expected inequality. By \eqref{ass:PhiWeakSeqCts},  $u_\rho-\Phi(u_\rho) \weaklyto u-\Phi(u)$ in $V$, which in conjunction with the fact that $\sigma_\rho(u_\rho-\Phi(u_\rho)) \to 0$ implies by \cref{lem:penaltyTypeCondition} that $u \leq \Phi(u)$ so that
$u$ solves \eqref{eq:QVIIntro}.

We also have, using \cref{lem:comparisonRhoAndVI} for the first inequality and with the limits below being weak,
\[u_\rho = \M_\rho(f) = \lim_n \overline u^n_\rho \geq \lim_n \hat u^n \geq \M(f) \]
(recall $\hat u^n = S(f,\hat u^{n-1})$ is defined above) where we used \cref{lem:convergenceOfVIsToQVIs} for the final inequality. Passing to weak limit in $\rho$ proves the result.

For strong convergence, we begin by defining $v_\rho := u+\Phi(u_\rho)-\Phi(u)$ which satisfies
\begin{align*}
&v_\rho \to u \text{ in $V$},\\
&v_\rho \leq \Phi(u_\rho),\\
&u_\rho - v_\rho = (u_\rho -u) + (\Phi(u)-\Phi(u_\rho)) \weaklyto 0 \text{ in $V$},
\end{align*}
with the first line thanks to \eqref{ass:PhiCC}. Testing the equation for $u_\rho$ with $u_\rho-v_\rho$, we have
\[\langle A(u_\rho - v_\rho), u_\rho-v_\rho \rangle = \langle f, u_\rho -v_\rho \rangle - \frac 1\rho \langle \sigma_\rho(u_\rho-\Phi(u_\rho)), u_\rho - v_\rho\rangle  - \langle Av_\rho, u_\rho - v_\rho \rangle\]
and to this we apply the monotonicity formula \eqref{eq:monotonicityIneq} and coercivity of $A$ to find
\begin{align*}
C_a\norm{u_\rho - v_\rho}{V}^2 &\leq \langle f, u_\rho -v_\rho \rangle - \frac 1\rho \langle  \sigma_\rho(v_\rho-\Phi(u_\rho)), u_\rho - v_\rho\rangle - \langle Av_\rho, u_\rho - v_\rho \rangle\\
&= \langle f, u_\rho -v_\rho \rangle- \langle Av_\rho, u_\rho - v_\rho \rangle\tag{since $v_\rho \leq \Phi(u_\rho)$}.
\end{align*}
The right-hand side converges to zero, hence $u_\rho-v_\rho \to 0$ strongly in $V$, implying $u_\rho \to u$.

The claim for the minimal solution follows similar lines as the above to deduce that the weak limit $u$
is a solution to \eqref{eq:QVIIntro}. 
\end{proof}
Note that we are not (yet) able to identify $u$ above as the minimal solution and prove that $\m_\rho(f) \weaklyto \m(f)$ in the above theorem under such general circumstances. it appears difficult to do this because $\underline u_\rho^n \geq \tilde u^n$ (by virtue of $\underline u_\rho^n$ being a solution of the unconstrained problem) and thus in the limit we do not obtain anything useful. This has been an open problem as identified in \cite[Chapter 4, Remark 1.4]{LionsBensoussan}. But we can identify the desired limit with a contractive argument under different assumptions, see \cref{thm:continuityOfRhoMapsUnderContraction}, which we will prove below.

	\subsection{The $\rho \searrow 0$ limit under a contraction assumption}\label{sec:convUnderContraction}

By assuming the small Lipschitz constant assumption \eqref{ass:PhiSmallLipschitzZ}, we can prove the convergence result that we wanted. It is convenient to introduce the following notation. Considering the cases $\Z=\m$ or $\Z = \M$, we define the mappings $Z^n \colon \R^+ \cup \{0\} \times V^* \to V$
via
\begin{align*}
	Z^{n}(\rho, g) &:= T_\rho( g, Z^{n-1}(\rho, g)),\\
	Z^0(\rho, g ) &:= \begin{cases}
	\underline u &\text{if $\Z = \m$}\\
	\overline u &\text{if $\Z = \M$}
	\end{cases}
	\end{align*}
where $\underline u, \overline u$ are given and independent of $\rho$ and $g$ (to be fixed later).
For convenience, we define $T_0(g,\varphi) := S(g,\varphi)$. With this, note that 	$Z^n(0,g) = S(g, Z^{n-1}(0,g))$. 
\begin{lem}\label{lem:ZnCts}
If $\Phi\colon V \to V$ is continuous, then for each $n$, the map $Z^n\colon \mathbb{R}^+ \times V^* \to V$ is continuous at all points from the set $\{0\} \times V^*$.
\end{lem}
\begin{proof}
We show this by induction. It is clear for $n = 0$ as $Z^0$ is constant in its arguments.  Assume that $Z^n(\rho,g) \to Z^n(0,f)$ as $(\rho,g) \to (0,f)$. The inductive step is:
		\begin{align*}
			Z^{n+1}(\rho, g) - Z^{n+1}(0,f)
			&=
			T_\rho( g, Z^n(\rho, g))
			-
			T_0( f, Z^n(0, f))
			\\&=
			[
				T_\rho( g, Z^n(\rho, g))
				-
				T_\rho( f, Z^n(0,f))
			]
			+
			[
				T_\rho( f, Z^n(0,f))
				-
				T_0( f, Z^n(0, f))
			]
			.
		\end{align*}
		The first bracket is continuous due to
		\cref{lem:TrhoLipschitz} and the induction hypothesis
		and the second bracket is continuous due to  \cref{lem:PDEApproximatesVI}.
\end{proof}

\begin{lem}\label{lem:ZnInBall}
Assume that $\Phi \colon V \to V$ is continuous
and $f \in V^*$ is such that \eqref{ass:PhiSmallLipschitzZ} and 
\begin{align}
&\text{$Z^n(0,f) \to \Z(f)$}\label{ass:gerd1}.
\end{align}
Then for every $\epsilon > 0$, there exist $N \in \mathbb{N}$
and $\rho_0,\delta > 0$ such that 
\begin{equation}
	\label{eq:ZnInBall}
	Z^n(\rho,g) \in B_\epsilon(\Z(f))
	\qquad\forall
	n \ge N, \rho \in [0,\rho_0], g \in B_\delta(f)
	.
\end{equation}
\end{lem}
Regarding the above assumptions, note that we are implicitly assuming that $\Z(f)$ is defined (it would be true if $f$, $\underline u$ and $\overline u$ satisfy \cref{ass:forZSubAndSuperSolutions},  see \cref{lem:convergenceOfVIsToQVIs}).
\begin{proof}
Without loss of generality, we assume $\epsilon \le \epsilon^*$. Let us first record some useful estimates. 
	Due to \eqref{ass:gerd1}, we get $N \in \N$ such that $n \ge N$ gives $ Z^n(0,f) \in B_{\epsilon\slash 2}(\Z(f))$.
	For $g \in V^*$ and $\rho \ge 0$
	we have
\begin{align*}
		\norm{ Z^N(\rho, g) - \Z(f) }{V}
		&\le
		\norm{ Z^N(\rho, g) - Z^N(0,f) }{V}
		+
		\norm{ Z^N(0,f) - \Z(f) }{V}\\
		&\le
		\norm{ Z^N(\rho, g) - Z^N(0,f) }{V}
		+
		\frac\epsilon2.
\end{align*}	
The function $Z^N$ is continuous at $(0,f)$ due to the previous lemma. This, together with \cref{lem:PDEApproximatesVI}, means that we can choose
	$\rho_0 > 0$
	and
	$\delta > 0$
	such that
	\begin{align*}
		Z^N(\rho,g) &\in B_\epsilon(\Z(f))
		\qquad\forall
		\rho \in [0,\rho_0], g \in B_\delta(f),
		\\
		\norm{T_\rho( f, \Z(f) ) - T_0(f, \Z(f))}{V}
		&\le
		\frac{1-\hat c}{2}\epsilon
		\mspace{53.5mu}
		\forall \rho \in [0,\rho_0],
		\\
		\delta
		&\le
		\frac{1 - \hat c}{1 + \hat C} \frac\epsilon2.
	\end{align*}
	Here,
	the constants $\hat C \ge 0$ and $\hat c \in [0,1)$
	are chosen as in \cref{lem:TrhoLipschitz}.

	By induction over $n$, we show that \eqref{eq:ZnInBall} holds.
	To this end,
	suppose that
	$n \ge N$, $\rho \in [0,\rho_0]$
	and $g \in B_\delta(f)$
	are given 	such that
	$Z^n(\rho,g) \in B_\epsilon(\Z(f))$.
	Using the definition of $Z^{n+1}$, we find
	\begin{align*}
		\norm{Z^{n+1}(\rho,g) - \Z(f) }{V}
		&=
		\norm{T_\rho( g, Z^n(\rho,g)) - T_0(f, \Z(f))}{V}
		\\&
		\le
		\norm{T_\rho( g, Z^n(\rho,g)) - T_\rho(f, \Z(f))}{V}
		+
		\norm{T_\rho( f, \Z(f) ) - T_0(f, \Z(f))}{V}
		\\&
		\le
		\norm{T_\rho( g, Z^n(\rho,g)) - T_\rho(f, \Z(f))}{V}
		+
		\frac{1-\hat c}{2}\epsilon
		.
	\end{align*}
	Since $Z^n(\rho,g) \in B_\epsilon(\Z(f)) \subset B_{\epsilon^*}(\Z(f))$,
	we can apply \cref{lem:TrhoLipschitz}
	and get
	\begin{align*}
		\norm{Z^{n+1}(\rho,g) - \Z(f) }{V}
		&\le
		\hat C \norm{f - g}{V^*} + \hat c \norm{ Z^n(\rho,g) - \Z(f) }{V}
		+
		\frac{1-\hat c}{2}\epsilon\\
		&\le
		\hat C \delta
		+
		\hat c \epsilon
		+
		\frac{1-\hat c}{2}\epsilon\\
		&\le
		\epsilon
		.
	\end{align*}
	This shows $Z^{n+1}(\rho, g) \in B_\epsilon(\Z(f))$.
	By induction, \eqref{eq:ZnInBall} follows.
\end{proof}
The important point in the previous result is that \eqref{eq:ZnInBall} holds for $\rho \leq \rho_0$ and $g \in B_\delta(f)$ uniformly in $n$.

Let us now prove the theorem on the convergence of $\Z_\rho(g) \to \Z(f)$.
\begin{proof}[Proof of \cref{thm:jointContinuityOfRhoMapsUnderContraction}]\label{proof:lipschitzZrho}
We take $\underline u$ and $\overline u$ in the definition of $Z^0$ to satisfy \cref{ass:subAndSupersolutionsForZRho}.  By \cref{ass:subAndSupersolutionsForZRho}, \cref{ass:subAndSuperSolutions} is satisfied for every source term in $B_{\bar \delta}(f)\cap W$  and we get that $\Z$ and $\Z_\rho$ are well defined on the set $B_{\bar \delta}(f)\cap W$ for small $\rho$.

\medskip

\noindent \emph{Step 1.} By \cref{lem:ZnInBall} (note that $\Phi$ is continuous by \eqref{ass:PhiCC}) there exist $N \in \mathbb{N}$ and $\rho_0, \delta > 0$ such that $Z^n(\rho,g) \in B_{\epsilon^*\slash 2}(\Z(f))$ as long as $n \geq N$, $\rho \leq \rho_0$ and $g \in B_{\delta}(f)$. Without loss of generality, we can assume that $\delta \leq \hat \delta$. For $g \in B_\delta(f)$, let us define the sequence 
\begin{align*}
y^n &:= T_\rho(g, y^{n-1})\\
y^0 &:= Z^N(\rho, g).
\end{align*}
i.e., $y^n = Z^{N+n}(\rho, g)$. As noted, we have $y^0 \in B_{\epsilon^*\slash 2}(\Z(f))$ under the stated conditions on $\rho$ and $g$. We claim that $T_\rho(g,\cdot) \colon B_{\epsilon^*\slash 2}(\Z(f)) \to  B_{\epsilon^*\slash 2}(\Z(f))$ is a contraction for sufficiently small $\rho$ and $g$ sufficiently close to $f$. Indeed, take $\varphi \in B_{\epsilon^*\slash 2}(\Z(f))$, $g \in B_{\delta^*}(f)$ 
where 
\[\delta^* = \frac{(1-\hat c)\epsilon^*}{4(\hat C + 1)}\] and take $\rho$ small enough (let us say $\rho \leq \rho_1$) so that $T_\rho(f, \Z(f)) \in B_{\hat C \delta^*}(\Z(f))$ (this is possible by \cref{lem:PDEApproximatesVI}).
Then using  \cref{lem:TrhoLipschitz} and \eqref{ass:PhiSmallLipschitzZ},
\begin{align*}
\norm{T_\rho(g,\varphi)-\Z(f)}{V}
&\leq \norm{T_\rho(g,\varphi) - T_\rho(f, \Z(f))}{V} + \norm{T_\rho(f,\Z(f)) - \Z(f)}{V}\\
&\leq \hat C\norm{g-f}{V^*} + \hat c\norm{\varphi-\Z(f)}{V} + \frac{(1-\hat c)\epsilon^*}{4} \\
&\leq \frac{(1-\hat c)\epsilon^*}{4} + \frac{\hat c \epsilon^*}{2} + \frac{(1-\hat c)\epsilon^*}{4}\\
&= \frac{\epsilon^*}{2}
\end{align*}
so that $T_\rho(g,\cdot)$ is invariant on the ball in question. 
For the contraction property, due to \cref{lem:TrhoLipschitz} and \eqref{ass:PhiSmallLipschitzZ},
\[\norm{T_\rho(g,\varphi)-T_\rho(g,\psi)}{V} \leq \hat c\norm{\varphi-\psi}{V} \qquad \forall \varphi, \psi \in B_{\epsilon^*\slash 2}(\Z(f)) \]
where $\hat c \in [0,1)$.  

Hence, by Banach's fixed point theorem, we obtain $y^n \to y$ in $V$ where $y=T_\rho(g,y)$. That is, $Z^n(\rho,g) \to Z^\infty(\rho,g)$ for some $Z^\infty(\rho,g)$. Furthermore, we have
\[\norm{y^n-y}{V} \leq \hat c \norm{y^{n-1}-y}{V}\]
which implies
\[\norm{y^n-y}{V} \leq \hat c^n \norm{y_0-y}{V}\]
i.e.
\[\norm{Z^{N+n}(\rho,g)-Z^\infty(\rho,g)}{V} \leq \hat c^n \norm{Z^N(\rho,g)-Z^\infty(\rho,g)}{V} \leq \epsilon \hat c^n\]
since $Z^N(\rho,g), Z^\infty(\rho,g) \in B_{\epsilon\slash 2}(\Z(f))$. Recall that the above holds as long as $n \geq N$, $\rho \leq \min(\rho_0, \rho_1)$ and $g \in B_{\min(\delta, \delta^*)}(f)$.

We can rewrite this as
\[\norm{Z^{n}(\rho,g)-Z^\infty(\rho,g)}{V} \leq \epsilon \hat c^{n-N}\quad\text{for $n \geq 2N$, $\rho \leq \min(\rho_0, \rho_1)$ and $g \in B_{\min(\delta, \delta^*)}(f)$},\]
where we note that the right-hand side of the inequality is independent of $\rho$ and $g$.

\medskip

\noindent \emph{Step 2.} \cref{ass:subAndSupersolutionsForZRho} implies that \cref{ass:forZRhoSubAndSuperSolutions} is satisfied for all  $g \in B_{\bar \delta}(f) \cap W$ and small enough $\rho$ and thus for $g$ taken in $B_{\bar \delta}(f) \cap W$, we can apply \cref{prop:strongConvergenceOfPenVIProblems}  which allows us to identify $Z^\infty(\rho, g) = \Z_\rho(g)$.

\medskip

\noindent \emph{Conclusion.} 
		To summarise, we have shown
				\begin{equation*}
		\begin{aligned}
			Z^n(\rho, g) &\to \Z_\rho(g) &&\text{uniformly in $\rho$, $g \in B_{\delta}(f) \cap W$ as $n \to \infty$},\\
		\end{aligned}
		\end{equation*}
		while from \cref{lem:ZnCts}, we have 
				\begin{equation*}
						\begin{aligned}
Z^n(\rho, g) &\to Z^n(0,f) &&\text{as $\rho \searrow 0$, $g \to f$}.
		\end{aligned}
		\end{equation*}		
Thus, we can interchange the iterated limits and get (the limit $g \to f$ below should be understood for $g \in W$)
		\begin{equation*}
			\lim_{\substack{\rho \searrow 0\\g\to f}} \Z_\rho(g)
			=
			\lim_{\substack{\rho \searrow 0\\g\to f}} \lim_{n \to \infty} Z^n(\rho, g)
			=
			\lim_{n \to \infty} \lim_{\substack{\rho \searrow 0\\g\to f}} Z^n(\rho, g)
			=
			\lim_{n \to \infty} Z^n(0,f)
			=
			\Z(f).
		\end{equation*}
\end{proof}

By taking $\underline u$, $\overline u$ in the definition of $Z^0$ to satisfy  \cref{ass:subAndSuperSolutions} (rather than \cref{ass:subAndSupersolutionsForZRho}) and arguing similarly to above, we obtain \cref{thm:continuityOfRhoMapsUnderContraction}.

\begin{remark}\label{remark:constructiveApproxOfExtremals}
Examining \cref{sec:penalisedProblem} and \cref{sec:convergencetoQVIs}, we see that there is a constructive way to approach the minimal and maximal solutions: we start at a subsolution or a supersolution, solve iteratively to get $\underline{u}^n_\rho$ or $\overline{u}^n_\rho$ (see \cref{defn:iterativeSeq}) for a large $n$, take $\rho$ small and we will be close to the minimal solution or the maximal solution, thanks to the results of \cref{prop:strongConvergenceOfPenVIProblems} and either \cref{thm:convergenceOfMRho}  or, in the case of the maximal solution,  \cref{thm:continuityOfRhoMapsUnderContraction}. This can be useful for numerical realisations.
\end{remark}
\section{Local Lipschitz continuity of $\Z$ and $\Z_\rho$}\label{sec:lipschitz}
Local Lipschitz continuity for these maps does not immediately follow from the continuous dependence estimate of \cref{lem:TrhoLipschitz} if we impose only the local Lipschitz condition of $\Phi$ (as in the statement of the result below), since we do not know a priori that (even if $f$ and $g$ are close enough) $\Z_\rho(f)$ and $\Z_\rho(g)$ are in the neighbourhood where $\Phi$ is Lipschitz with a small Lipschitz constant. Instead, we have to argue using the results of the above section.

\begin{proof}[Proof of \cref{thm:localLipschitzForZRho}]
Since by \cref{thm:jointContinuityOfRhoMapsUnderContraction}, $\Z_\rho(g) \in B_{\epsilon^*}(\Z(f))$ for all $g \in W$ sufficiently close to $f$ and $\rho$ sufficiently small, we obtain via \cref{lem:TrhoLipschitz} and \eqref{ass:PhiSmallLipschitzZ}, for $\hat c < 1$, the estimate
\[\norm{\Z_\rho(f)-\Z_\rho(g)}{V} \leq \hat C\norm{f-g}{V^*} + \hat c\norm{\Z_\rho(f)-\Z_\rho(g)}{V}.\]
\end{proof}
Regarding Lipschitz continuity for $\Z$, a first thought might be that we could pass to the limit in $\rho$ in the inequality of \cref{thm:localLipschitzForZRho} but the assumptions with respect to $T_\rho$ would still be needed with that approach.  We argue differently.

\begin{proof}[Proof of Theorem \ref{thm:localLipschitzForZ}]\label{proof:lipschitzZ}

The idea is that if we had that $\Z(g) \in B_{\epsilon^*}(\Z(f))$ for $g$ sufficiently close to $f$, we can, like in the above proof, once again apply \cref{lem:TrhoLipschitz} (with $\rho=0$) and the smallness assumption \eqref{ass:PhiSmallLipschitzZ} to obtain the result.

Thus, we need the result of \cref{thm:jointContinuityOfRhoMapsUnderContraction} for $\rho=0$ (without any assumptions on $\sigma_\rho$ or other $\rho$-dependent quantities. This can be achieved by simply noting that the arguments of \cref{sec:convUnderContraction} still hold with $\rho=0$ and with \cref{ass:subAndSupersolutionsForZRho} replaced by \cref{ass:subAndSupersolutionsZ}. The proofs of the results can be modified in the obvious way but let us point out that in the proof of \cref{thm:jointContinuityOfRhoMapsUnderContraction}, we need to use  \cref{lem:convergenceOfVIsToQVIs} in place of \cref{prop:strongConvergenceOfPenVIProblems} (observe that \eqref{ass:aAssumptionsForVI} implies \cref{ass:forZSubAndSuperSolutions}).
\end{proof}
It is worth noting that the Lipschitz constants in \cref{thm:localLipschitzForZ} and \cref{thm:localLipschitzForZRho} are both exactly 
\[\frac{\hat C}{1-\hat c}\]
with $\hat c$ and $\hat C$ given in \cref{lem:TrhoLipschitz}.
\section{Directional differentiability}\label{sec:dirDiff}
In this section, we shall prove that $\Z_\rho$ and $\Z$ are directionally differentiable maps (and also Hadamard differentiable in a certain sense). Our line of attack is based on the iteration approach from \cite{AHR} (where we approximate the QVI solutions by a sequence of solutions of VIs, derive expansion formula for the elements of the sequence and then pass to the limit) combined with some refinements in \cite{Wachsmuth2021:2}. We start with the analysis for $\Z_\rho$.

\subsection{Differentiability for $\Z_\rho$}\label{sec:rhoPositiveCase}

An essential task is to obtain differentiability for $T_\rho$ in its arguments. In the equation defining $T_\rho$, observe that the nonlinearity $\sigma_\rho\colon V \to V^*$ is Hadamard differentiable and the derivative is bounded in the direction: in fact, when seen as a real-valued function, $\sigma_\rho$ is $C^1$, and by using the Lipschitzness and boundedness of $\sigma_\rho'$, we have that $\sigma_\rho\colon V \to V^*$ is G\^ateaux differentiable \cite[Theorem 8]{MR1231260} (thus, it is also Hadamard differentiable since $\sigma_\rho$ is Lipschitz). We use this fact below.
\begin{lem}\label{lem:TRhoDirectionallyDiff}
Let  $f \in V^*$ and assume that $\Phi$ is directionally differentiable at $\varphi \in V$. 
Then $T_\rho\colon V^* \times V \to V$ is directionally differentiable at $(f,\varphi)$, i.e.,
\[\lim_{s \searrow 0}\frac{T_\rho(f+sd,\varphi+sh)-T_\rho(f,\varphi)}{s} = T_\rho'(f,\varphi)(d,h)\qquad\text{for $d \in V^*$ and $h  \in V,$}\]
where $T_\rho'(f,\varphi)(d,h)=\delta$ is the unique solution of the equation  
\begin{equation}
	A\delta + \frac{1}{\rho}\sigma_\rho'(u-\Phi(\varphi))(\delta-\Phi'(\varphi)(h)) = d\label{eq:formulaForTRhoDerivative}.
\end{equation}
\end{lem}
\begin{proof}
First, it is easy to see that \eqref{eq:formulaForTRhoDerivative} has a unique solution: if we make a transformation $\hat \delta = \delta-\Phi'(\varphi)(h)$, we have
\[	A\hat \delta + \frac{1}{\rho}\sigma_\rho'(u-\Phi(\varphi))(\hat \delta) = d - A\Phi'(\varphi)(h)\]
and this is uniquely solvable by the Lax--Milgram lemma because the
linear operator $A + \frac{1}{\rho}\sigma_\rho'(u-\Phi(\varphi))$
is coercive and bounded\footnote{In general, if $\sigma_\rho$ is directionally differentiable at $w \in V$, then $\sigma_\rho'(w)\colon V \to V^*$ is a monotone operator; this follows from
\begin{align*}
\langle \sigma_\rho'(w)(a) - \sigma_\rho'(b), a-b \rangle &= 
\frac{1}{s}\langle \sigma_\rho(w+sa)-o^m(s,a) - \sigma_\rho(w+sb) + o^m(s,b), a-b \rangle \\
&\geq \frac{1}{s}\langle o^m(s,b)  - o^m(s,a), a-b \rangle.
\end{align*}}.

Let $y:=T_\rho(f+sd , \varphi + sh )$,  $u:=T_\rho(f,\varphi)$ and define $\delta$ as the solution of \eqref{eq:formulaForTRhoDerivative}.  Let us make the transformation $\hat y = y-\Phi(\varphi+sh)$, $\hat u = u-\Phi(\varphi)$ and (as above) $\hat \delta = \delta-\Phi'(\varphi)(h)$ so that
\begin{align*}
A\hat y + \frac{1}{\rho}\sigma_\rho(\hat y) &= f+sd  - A\Phi(\varphi + sh ),\\
A\hat u + \frac{1}{\rho}\sigma_\rho(\hat u) &= f - A\Phi(\varphi),\\
A\hat \delta + \frac{1}{\rho}\sigma_\rho'(\hat u )(\hat \delta) &= d - A\Phi'(\varphi)(h).
\end{align*}
Multiplying by $s$ the last equation and subtracting the latter two equations from the first and adding and subtracting $\rho^{-1}\sigma_\rho(\hat u+ s\hat \delta)$, we obtain
\begin{align*}
&A(\hat y-\hat u-s\hat \delta) + \frac{1}{\rho}\left(\sigma_\rho(\hat y) - \sigma_\rho(\hat u+s\hat \delta)\right) + \frac{1}{\rho}\left(\sigma_\rho(\hat u+s\hat \delta) - \sigma_\rho(\hat u) - s\sigma_\rho'(\hat u)(\hat \delta)\right) = - A l_s(\varphi, h),
\end{align*}
where $l_s$ is the remainder term associated to $\Phi$. The above is, using the fact that $\sigma_\rho$ is directionally differentiable, 
\begin{align*}
&A(\hat y-\hat u-s\hat \delta) + \frac{1}{\rho}\left(\sigma_\rho(\hat y) - \sigma_\rho(\hat u+s\hat \delta)\right) + \frac{1}{\rho}o_s^m(\hat u, \hat \delta) =- A l_s(\varphi, h),
\end{align*}
where $o_s^m$ denotes the remainder term of $\sigma_\rho$. Testing with $\hat y-\hat u-s\hat \delta$ and using monotonicity,
	\[C_a\norm{\hat y- \hat u - s\hat \delta}{V} \leq \frac 1\rho \norm{o_s^m(\hat u, \hat \delta)}{V^*} + C_b\norm{ l_s(\varphi, h)}{V}.\]
Now note that $\hat y- \hat u - s\hat \delta = y-u-s\delta - l_s(\varphi, h)$ so that
\begin{align*}
C_a\norm{y-  u - s\delta}{V} &\leq \frac 1\rho \norm{o_s^m(u-\Phi(\varphi), \delta-\Phi'(\varphi)(h))}{V^*} +  (C_a+C_b)\norm{ l_s(\varphi, h)}{V}.
\end{align*}
Dividing by $s$ and sending $s \to 0$ proves the result.
\end{proof}
Suppose that we are 
\[\text{given $f \in V^*$ and a set $W \subseteq V^*$ satisfying \cref{ass:subAndSupersolutionsForZRho}}\]
so that $\Z_\rho(g)$ and $\Z(g)$ are well defined for all $g \in B_{\bar \delta}(f) \cap W$ and sufficiently small $\rho.$ 
Let $d \in \mathcal{T}_W(f)$, so there exist $\{d_k\}$ with $d_k \to d$ in $V^*$ and $\{s_k\}$ with $s_k \searrow 0$ such that $f+s_kd_k \in W$. Define
\begin{align*}
u_n^k &:= T_\rho(f+s_kd_k, u_{n-1}^k),\\
u_0^k &:= \Z_\rho(f).
\end{align*}
For convenience, let us also define
\[u:=\Z_\rho(f).\]
We have omitted writing the dependence on $\rho$ in these definitions for ease of reading. In the following, we need, in particular, that $\Phi$ is locally Lipschitz on $B_{\epsilon^*}(\Z(f))$ and take $C_L$ as in \eqref{ass:generalLipschitzConstant} from \cref{lem:small_estimate}, i.e., we assume \eqref{ass:PhiSmallLipschitzZ}.  An alternative approach could be to instead assume it is locally Lipschitz on $B_{\epsilon^*}(\Z_\rho(f))$; this would entail a different set of assumptions to below.

\begin{lem}\label{lem:iteratesInBallNEW}
Assume
\eqref{ass:PhiCC},
\eqref{ass:PhiSmallLipschitzZ} and \cref{ass:subAndSupersolutionsForZRho}. 
If $\rho$ is sufficiently small and $k$ is sufficiently large, we have 
\[u_n^k \to u^k := \Z_\rho(f + s_k d_k) \text{ in $V$ as $n \to \infty$}.\]
\end{lem}
\begin{proof}
	We take $\rho$ small enough and $K>1$ (to be specified later)
	such that $u_0^k = u = \Z_\rho(f) \in B_{{\epsilon^*}\slash K}(\Z(f)) \subset B_{\epsilon^*}(\Z(f))$,
	possible thanks to \cref{thm:continuityOfRhoMapsUnderContraction}
	\footnote{We could instead apply  \cref{thm:convergenceOfMRho} if $\Z_\rho = \M_\rho$ is under consideration; this would lead to different assumptions being required for this result.}.

	If $k$ is sufficiently large, we have $f+s_k d_k \in B_{\bar\delta}(f)$,
	and we have by assumption that $f+s_k d_k \in W$.
	Hence, by \cref{ass:subAndSupersolutionsForZRho},
	for $\rho$ sufficiently small and $k$ sufficiently large, $\Z_\rho(f+s_k d_k)$ is well defined
	and
	$\Z_\rho(f+s_kd_k) \in B_{\epsilon^*}(\Z(f))$,
	due to the local Lipschitz property for $\Z_\rho$, see \cref{thm:localLipschitzForZRho}.

	We next show that
	the operator $T_\rho( f + s_k d_k, \cdot)$ maps the ball $B_{\epsilon^*}(\Z(f))$
	onto itself,
	if $k$ and $K$ are large enough.
	We take an arbitrary $\varphi \in B_{\epsilon^*}({\Z(f)})$.
	By using
	$\Z_\rho(f) = T_\rho(f, \Z_\rho(f))$
	and by
	utilising \cref{lem:TrhoLipschitz},
	we get
	\begin{align*}
		\norm{T_\rho(f+s_kd_k, \varphi) - \Z(f)}{V}
		&\leq \norm{T_\rho(f+s_kd_k, \varphi) - \Z_\rho(f)}{V} + \norm{\Z_\rho(f)-\Z(f)}{V}\\
		&\leq s_k\hat C \norm{d_k}{V^*} + \hat c\norm{\varphi-\Z_\rho(f)}{V} + \frac{{\epsilon^*}}{K}.
	\end{align*}
	Here, $\hat C \ge 0$ and $\hat c \in [0,1)$ are given by \cref{lem:TrhoLipschitz}.
	For the second term on the right-hand side,
	we employ the triangle inequality to get
	\begin{equation*}
		\norm{\varphi-\Z_\rho(f)}{V}
		\le
		\norm{\varphi-\Z(f)}{V}
		\norm{\Z(f)-\Z_\rho(f)}{V}
		\le
		\epsilon^* + \frac{\epsilon^*}{K}.
	\end{equation*}
	Altogether,
	we arrive at
	\begin{equation*}
		\norm{T_\rho(f+s_kd_k, \varphi) - \Z(f)}{V}
		\leq
		s_k\hat C C_1 + \hat c\left({\epsilon^*} + \frac{{\epsilon^*}}{K}\right)+ \frac{{\epsilon^*}}{K}
	\end{equation*}
	where $C_1$ is the uniform bound on the $d_k$.
	The right-hand side is less than ${\epsilon^*}$ if $k$ is sufficiently large
	and $K$ is chosen large enough (we need $K > (1+\hat c)(1-\hat c)^{-1}$). 

	This proves the mapping property
	$T_\rho(f+s_kd_k, \cdot)\colon B_{\epsilon^*}(\Z(f)) \to B_{\epsilon^*}(\Z(f))$.
	Using \cref{lem:TrhoLipschitz} again, we find $T_\rho(f+s_kd_k, \cdot)$ is a contraction on $B_{\epsilon^*}(\Z(f))$.
	Hence, the assertions follow from the celebrated Banach fixed point theorem,
	since $\Z_\rho(f + s_k d_k)$
	is a fixed point of $T_\rho(f + s_k d_k; \cdot)$
	on $B_{\epsilon^*}(\Z(f))$.
\end{proof}

The next proposition shows that if $\Phi$ is differentiable at $u=\Z_\rho(f)$, we can obtain a Taylor expansion for $u_n^k$.
	\begin{prop}\label{prop:unsDifferentiableNEW}
Let \eqref{ass:PhiSmallLipschitzZ} and  \eqref{ass:PhiDiffAtU}  hold. 
For $\rho$ sufficiently small, we have for each $n$, 
	\begin{equation*}
		\lim_{k \to \infty}\frac{u_n^k - u}{s_k} = \alpha_n
	\end{equation*}						
	where  $\alpha_n := T_\rho'(f,u)(d,\alpha_{n-1})$, i.e.,
		\[A\alpha_n + \frac{1}{\rho}\sigma_\rho'(u-\Phi(u))(\alpha_n-\Phi'(u)(\alpha_{n-1})) = d.\]
	\end{prop}
	\begin{proof}
First of all, due to \cref{lem:TrhoLipschitz} (which gives local Lipschitzness for $T_\rho$ around $V^*\times B_{\epsilon^*}(\Z(f))$) and \cref{lem:TRhoDirectionallyDiff} (which gives directional differentiability of $T_\rho$ at $(f,u)$) we find that $T_\rho$ is Hadamard differentiable at $(f,u)$ because we have taken $\rho$ such that $u \in B_{\epsilon^*}(\Z(f))$.
	
We use a proof by induction. The base case is obviously true (with $\alpha_n  = 0$).
Assume $(1\slash s_k)(u_n^k-u) \to \alpha_n$. Then we have
\begin{align*}
\frac{u_{n+1}^k-u}{s_{k}} &= \frac{T_\rho(f+s_{k}d_{k}, u_n^k)-T_\rho(f,u)}{s_{k}} \\
&= \frac{T_\rho(f+s_{k}d_{k}, u+s_k(\frac{u_n^k-u}{s_k}))-T_\rho(f,u)}{s_{k}}\\
&\to T_\rho'(f,u)(d,\alpha_n)
\end{align*}
where we used that
$T_\rho$ is Hadamard differentiable and $d_k \to d$.
	\end{proof}

\begin{lem}\label{lem:alphaNBoundedNEW}
Let \eqref{ass:PhiSmallLipschitzZ} and \eqref{ass:PhiDiffAtU} hold. We have that $\alpha_n \to \alpha$ in $V$, where $\alpha$ is the unique solution of
	\[A\alpha + \frac{1}{\rho}\sigma_\rho'(u-\Phi(u))(\alpha-\Phi'(u)(\alpha)) = d.\]
Furthermore, the map $d \mapsto \alpha$ is bounded and continuous from $V^*$ to $V$.
\end{lem}
\begin{proof}
Consider the map $\beta \mapsto \alpha$ defined
as the solution mapping of
	\[A\alpha + \frac{1}{\rho}\sigma_\rho'(u-\Phi(u))(\alpha-\Phi'(u)(\beta)) = d,\]
	i.e., the map $T_\rho'(f,u)(d, \cdot)$.
	We show that it is a contraction.
	By using $\beta, \hat\beta \in V$
	and the associated solutions $\alpha, \hat\alpha \in V$,
	we get
	\[A(\alpha-\hat\alpha) +\frac{1}{\rho}\left(\sigma_\rho'(u-\Phi(u))(\alpha- \Phi'(u)(\beta))  - \sigma_\rho'(u-\Phi(u))(\hat\alpha -\Phi'(u)(\hat\beta))\right) = 0.\]
Testing with $\alpha- \Phi'(u)(\beta) -\hat\alpha + \Phi'(u)(\hat\beta)$ and using monotonicity, we obtain
	\[\langle A(\alpha-\hat\alpha), \alpha- \Phi'(u)(\beta) -\hat\alpha + \Phi'(u)(\hat\beta)\rangle \leq 0.\]	
	Now, since $\Phi'(u)\colon V \to V$ is Lipschitz with the same Lipschitz constant $C_L$ as $\Phi$, we obtain via similar arguments to \cite{Wachsmuth2021:2} (see also the proof of \cref{lem:TrhoLipschitz} following \eqref{eq:tested_VIs}) that $T_\rho'(f,u)(d, \cdot)$ is a contraction since $C_L$ satisfies \eqref{ass:generalLipschitzConstant}. The Banach fixed point theorem gives the result.
	
The map $d \mapsto \alpha$ defined through \eqref{eq:equationForAlphaRhoNEW} is sensible for all $d \in V^*$ by the above procedure, and it is bounded as can be seen by testing with $\alpha-\Phi'(u)(\alpha)$ and using the smallness condition on $C_L$ of \cref{lem:small_estimate}.
For continuity, if $d_n \to d$ in $V^*$ and $\alpha_n$ and $\alpha$ are the associated derivatives, we have
	\[\langle A(\alpha_n-\alpha), \alpha_n- \Phi'(u)(\alpha_n) -\alpha + \Phi'(u)(\alpha)\rangle \leq \langle d_n -d, \alpha_n- \Phi'(u)(\alpha_n) -\alpha + \Phi'(u)(\alpha)\rangle\]
and making use again of the Lipschitz property of $\Phi'(u)$, we conclude the claim from
\[\norm{\alpha_n - \alpha}{V} \leq C\norm{d_n-d}{V^*}.\]	
\end{proof}
\begin{lem}\label{lem:onsUniformNEW}Let the assumptions of \cref{thm:localLipschitzForZRho} hold. If $k$ is sufficiently large and $\rho$ is sufficiently small, we have
\[\lim_{n \to \infty}\frac{u^k_n -u}{s_k} =\frac{u^k-u}{s_k}\quad \text{ uniformly in $k$ and $\rho$.}\]
\end{lem}
In \cite[\S 5.3]{AHR}, three of the present authors showed that (under a different setup to what we have here) the limit $\lim_{s \searrow 0} \frac{u_n^s-s}{s}$ is uniform in $n$. Here though, like in \cite[Theorem 32]{Wachsmuth2021:2}, we will show uniformity in $k$ (and $\rho$) in the limit $n \to \infty$.
\begin{proof}
We argue similarly to \cite[Proof of Theorem 32]{Wachsmuth2021:2}.
In the proof of \cref{lem:iteratesInBallNEW}, we have used
the Banach fixed point theorem to obtain the convergence of the sequence $u_n^k$.
This directly yields the a-priori estimate
\begin{equation*}
	\norm{ u_n^k - u^k }{V}
	\le
	\hat c^k
	\norm{u^k_0-u^k}{V}
	\leq \hat c^n C s_k \norm{d_k}{V^*},
\end{equation*}
where we used that
$u^k_0 - u^k= \Z_\rho(f)-\Z_\rho(f+s_kd_k)$ and the estimate from \cref{thm:localLipschitzForZRho}.
This shows that
\[\frac{\norm{u_n^k-u^k}{V}}{s_k} \to 0 \quad\text{ as $n \to \infty$ uniformly in $k$ and $\rho$}.\]
Using  $u^k_n -u^k = u^k_n -u - (u^k-u)$, we deduce the result.
\end{proof}
We now prove our differentiability result for $\Z_\rho$.
\begin{proof}[Proof of \cref{thm:hadamardDiffMRhoNEW}]
The above results allow us to switch limits:
\begin{align*}
\lim_{k \to \infty} \frac{u^k-u}{s_k} = \lim_{k \to \infty} \lim_{n \to \infty} \frac{u_n^k-u}{s_k} =  \lim_{n \to \infty} \lim_{k \to \infty}\frac{u_n^k-u}{s_k} = \lim_{n \to \infty}\alpha_n = \alpha.
\end{align*}
This is exactly the Hadamard differentiability claim \ref{item:HadamardDifferentiabilityOfZRho}. 
The remaining assertions on the derivative have been shown in \cref{lem:alphaNBoundedNEW}. 
\end{proof}

\subsection{Differentiability for $\Z$}\label{sec:differentiabilityForZProof}
We cannot pass to the limit in $\rho$ to deduce that $\Z$ is differentiable because we do not have uniformity in $\rho$ or $s$ of the appropriate expression but we may repeat the arguments in \cref{sec:rhoPositiveCase} with $\rho$ taken to be zero and with  \cref{ass:subAndSupersolutionsZ} rather than \cref{ass:subAndSupersolutionsForZRho}. Let us point out the changes. For the $\rho=0$ version of \cref{lem:TRhoDirectionallyDiff}, we have from similar arguments to \cite[Propositon 1]{AHR} the following, making use of the differentiability of the VI solution map result in \cite{WGuidedTour} given under a general vector lattice setting, which generalises Mignot's result in \cite{MR0423155}.
\begin{lem}\label{lem:SDirectionallyDiff}Let  $\Phi$ be directionally differentiable at $\varphi \in V$   and take $f \in V^*$. Then $S\colon V^* \times V \to V$ is directionally differentiable at $(f,\varphi)$ and we have
\[\frac{S(f+sd,\varphi+sh)-S(f,\varphi)}{s} \to S'(f,\varphi)(d,h)\qquad\text{for $d \in V^*$ and $h  \in V$,}\]
where the derivative $S'(f,\varphi)(d,h)=\delta$ is the solution of the inequality
\begin{equation*}
\delta \in \mathcal{K}^u(\varphi, h) :	\langle A\delta-d, \delta-v \rangle \leq 0 \quad \forall v \in \mathcal{K}^u(\varphi, h)
\end{equation*}
where $u= S(f,\varphi)$ and
\[\mathcal{K}^u(\varphi,h) := \Phi'(\varphi)(h) + \mathcal{T}_{\mathbf K(\varphi)}(u) \cap [f-Au]^\perp.\]
\end{lem}
Above, recall that $\mathcal{T}_{\mathbf K(\varphi)}(u)$ is the tangent cone, which can be defined as the closure $\overline{\mathcal{R}_{\mathbf K(\varphi)}(u)}$.
\begin{remark}\label{rem:criticalConeExpression}
When we are in a Dirichlet space setting (see the discussion around \cref{ass:forCStationarity}), we obtain an explicit expression for the tangent cone and we  in fact have that
\[\mathcal{K}^u(\varphi,h) = \{ w \in V : w \leq \Phi'(\varphi)(h) \text{ q.e. on } \{u=\Phi(\varphi)\} \text{ and } \langle Au-f, w -\Phi'(\varphi)(h)\rangle = 0\}.\]
\end{remark}
 Let us assume \eqref{ass:PhiCC}, \eqref{ass:PhiSmallLipschitzZ}, \eqref{ass:PhiDiffAtUNew} and \cref{ass:subAndSupersolutionsZ}. Similarly to before, we let $d \in \mathcal{T}_W(f)$, so that there exist $\{d_k\}$ with $d_k \to d$ in $V^*$ and $\{s_k\}$ with $s_k \searrow 0$ such that $f+s_kd_k \in W$. Define
\begin{align*}
u_n^k &:= S(f+s_kd_k, u_{n-1}^k),\\
u_0^k &:= \Z(f),
\end{align*}
and $u=\Z(f)$.
\begin{itemize}
\item \cref{lem:iteratesInBallNEW} still holds with $u_n^k \to u^k := \Z(f+s_kd_k)$ if we use \cref{thm:localLipschitzForZ} instead of \cref{thm:localLipschitzForZRho}.  

\item In \cref{prop:unsDifferentiableNEW}, we may use \cref{lem:SDirectionallyDiff} instead of \cref{lem:TRhoDirectionallyDiff}  and we have instead that $\alpha_n := S'(f,u)(d, \alpha_{n-1})$ which satisfies $d-A\alpha_n \in \mathcal N_{\mathcal{K}^u(u, \alpha_{n-1})}(\alpha_n)$,
i.e.,
\[\alpha_n \in \mathcal{K}^u(u, \alpha_{n-1}) : \langle A\alpha_n - d, \alpha_n - v \rangle \leq 0 \qquad \forall v \in \mathcal{K}^u(u, \alpha_{n-1}).\]
\item The result of \cref{lem:alphaNBoundedNEW} still holds.  The map for the derivative is well defined for all $d \in V^*$: we can consider the VI
\[	 		\alpha \in \mathcal{K}^u(\beta) : \langle A\alpha -d, \alpha - v \rangle \leq 0 \quad \forall v \in \mathcal{K}^u(\beta)	 \]
and use a fixed point approach, just like in the proof of \cref{lem:alphaNBoundedNEW}  (or see \cite[Proposition 3.9]{AHROCQVI}). For the continuity of the derivative, a similar argument to that above works (or see \cite[Proposition 3.12]{AHROCQVI}). \cref{lem:onsUniformNEW} also holds if we again use \cref{thm:localLipschitzForZ}.
\item  Finally, arguing similarly to the proof of \cref{thm:hadamardDiffMRhoNEW},  we can prove \cref{thm:hadamardDiffZNEW}.
\end{itemize}

\section{Optimal control and stationarity}\label{sec:oc}

The proof for the existence of optimal points is straightforward.
\begin{proof}[Proof of \cref{thm:rfsfsdd}]

Let $\{f_n\} \subset \Uad$ be an infimising sequence with $y_n =\M(f_n)$ and $z_n = \m(f_n)$, i.e., \[J(y_n, z_n, f_n) \to \inf_{\substack{f \in \Uad,\\ y =\M(f),\\ z=\m(f)}}J(y,z, f).\]
Then by \cref{ass:smallLipschitzForOCProblem} \ref{item:uniformBd}, $\{f_n\}$ is bounded in $H$
and therefore, there exists $f^* \in H$ such that, for a subsequence, 
\[f_{n_j} \weaklyto f^* \text{ in $H$}.\]
The weak sequential closedness of $\Uad$  yields that $f^* \in \Uad$.  By \cref{ass:smallLipschitzForOCProblem} \ref{item:smallLipschitzM} and \ref{item:smallLipschitzm}, \eqref{ass:PhiSmallLipschitzZ}  holds in a ball around the points $\M(f^*)$ and $\m(f^*)$. Using $H \ctsCompact V^*$, we have $f_{n_j} \to f^*$ in $V^*$  so $f_{n_j} \in B_{\delta}(f^*)$ sufficiently far along the sequence.  

Since $B_{\bar \delta}(f) \cap \Uad  \subset \Uad$ for any $f \in \Uad$, by  \cref{ass:ocSubAndSupersolutions}, we have that \cref{ass:subAndSupersolutionsZ} holds  (with $W$ selected as $\Uad$). 
Thus we can use \cref{thm:localLipschitzForZ}, and pass to the limit to discover $(y_{n_j}, z_{n_j}) =(\M(f_{n_j}), \m(f_{n_j})) \to (\M(f^*), \m(f^*)) = (y^*, z^*)$ in $V$.

To see that this point is optimal, we observe that (dispensing with the subsequence notation now), using \cref{ass:smallLipschitzForOCProblem} \ref{item:wlscForJ},
\begin{align*}
J(y^*,z^*, f^*) &\leq \liminf_{n \to \infty} J(y_n, z_n, f_n) \leq \lim_{n \to \infty} J(y_n, z_n, f_n) = \min_{\substack{f \in \Uad\\ y = \M(f),\\ z =\m(f)}}J(y,z,f).\qedhere
\end{align*}
\end{proof}

\subsection{Bouligand stationarity}\label{sec:BS}
Working directly with the nonsmooth optimisation problem, we can obtain a Bouligand stationarity characterisation of local minimisers (as in the case for variational inequalities, see \cite[\S 5]{MR0423155} and \cite[Lemma 3.1]{MR739836}).

\begin{proof}[Proof of \cref{lem:characterisationOfOC}] Take $h$ in the radial cone of $\Uad$ at $f^*$ so that it is an admissible direction. Writing $y_s = \M(f^*+sh)$ and $z_s = \m(f^*+sh)$, we obtain by \cref{thm:hadamardDiffZNEW} that
\begin{align*}
y_s = y^* + s\alpha + o(s) \quad\text{and}\quad z_s = z^* + s\beta + o(s),
\end{align*}
where  $o$ is a remainder term and $\alpha = \M'(f^*)(h)$ and $\beta = \m'(f^*)(h)$. It follows that $(f^*+sh,y_s,z_s)$ can be made arbitrarily close to $(f^*,y^*,z^*)$ if $s$ is sufficiently small.

By definition of local minimiser, we have $J(y_s, z_s, f^* + sh) - J(y^*, z^*, f^*) \geq 0$ for $s$ sufficiently small. Dividing by $s$ and taking the limit, using the fact that $J$ is (at least) Hadamard differentiable, this yields
\[J_y(y^*,z^*,f^*)(\alpha) + J_z(y^*,z^*,f^*)(\beta) + J_f(y^*,z^*,f^*)(h) \geq 0 \quad \forall h \in \mathcal{R}_{\Uad}(f^*),\]
and by density and continuity of the derivatives appearing above with respect to the direction, also for $h \in \mathcal{T}_{\Uad}(f^*).$ 
\end{proof}
\subsection{The penalised problem}
We will not work directly with the penalised problem \eqref{eq:ocProblemMainResults} but instead a modified problem in order to prove that \textit{every} minimiser is a stationarity point. This is a classical localisation approach.
\begin{prop}\label{prop:convergenceOfOptimalPoints}
Assume \eqref{ass:PhiCC}.
For any local minimiser $(y^*,z^*,f^*)$ of \eqref{eq:ocProblemMostGeneral}, there exists a sequence of locally optimal points $(y_\rho^*, z_\rho^*, f_\rho^*)$ of  \begin{equation}\label{eq:ocProblemPenBarbu}
\min_{f \in \Uad} J(\M_\rho(f), \m_\rho(f),f) + \frac 12\norm{f-f^*}{H}^2
\end{equation}
 such that  $(y_\rho^*, z_\rho^*, f_\rho^*) \to (y^*,z^*, f^*)$ in $V \times V \times H$.
\end{prop}
\begin{proof}

Denote by $\gamma$ the radius such that $f^*$ is a minimiser on $\Uad\cap B^H_\gamma(f^*)$ (the latter object is the closed ball in $H$ of radius $\gamma$ with centre $f^*$). 

Define the augmented functional $\bar J(y, z, f) := J(y, z, f) + \frac 12\norm{f-f^*}{H}^2$ that appears in \eqref{eq:ocProblemPenBarbu} and consider the problem
\begin{equation}\label{eq:ocProblemPenBarbuA}
\min_{f \in \Uad \cap B_\gamma^H(f^*) } \bar J(\M_\rho(f), \m_\rho(f),f).
\end{equation}
By the same proof as for \cref{thm:rfsfsdd} with the obvious modifications, we find that there exists an optimal point to this problem, which we denote by $(\bar y_\rho, \bar z_\rho, \bar f_\rho)$. From
\begin{equation}\label{eq:2}
\bar J(\bar y_\rho,   \bar z_\rho, \bar f_\rho) \leq \bar J(\M_\rho(f^*), \m_\rho(f^*), f^*),
\end{equation}
and using $\Z_\rho(f^*) \to \Z(f^*)$ (due to \cref{thm:jointContinuityOfRhoMapsUnderContraction}) and the continuity of $\bar J$, we have
\[\limsup_{\rho \to 0}\bar J(\bar y_\rho, \bar z_\rho, \bar f_\rho) \leq J(y^*, z^*, f^*).\]
On the other hand,  it follows from \eqref{eq:2}  and \cref{thm:jointContinuityOfRhoMapsUnderContraction}
that $\bar J(\bar y_\rho, \bar z_\rho, \bar f_\rho)$ is uniformly bounded,  and hence, due to \cref{ass:smallLipschitzForOCProblem} \ref{item:uniformBd}, we obtain the existence of $\hat f$ such that (for a subsequence that we have relabelled) $\bar f_\rho \weaklyto \hat f$ in $H$ with the convergence strong in $V^*$. 

We have
\[\M_\rho(\bar f_\rho)-\M(\hat f) = (\M_\rho(\bar f_\rho)-\M_\rho(\hat f)) + (\M_\rho(\hat f)-\M(\hat f))\]
 and availing ourselves of the Lipschitz estimate of \cref{thm:localLipschitzForZRho} (with the Lipschitz constant independent of $\rho$), we have that the first term above converges to zero and the second term does also
 due to
 \cref{thm:continuityOfRhoMapsUnderContraction}.  Hence $\bar y_\rho \to \hat y:=\M(\hat f)$ and, arguing similarly, $\bar z_\rho \to \hat z :=\m(\hat f)$. By the identity $\limsup(a_n) + \liminf(b_n) \leq \limsup(a_n + b_n)$ and weak lower semicontinuity, this gives
\[\limsup_{\rho \to 0}\bar J(\bar y_\rho, \bar z_\rho, \bar f_\rho) \geq J(\hat y, \hat z, \hat f) + \limsup_{\rho \to 0}\norm{\bar f_\rho-f^*}{H}^2 \geq J(y^*, z^*, f^*) + \limsup_{\rho \to 0}\norm{\bar f_\rho-f^*}{H}^2, \]
with the last inequality because $(y^*, z^*, f^*)$ is a local minimiser and $\hat f \in B_\gamma^H(f^*)$. Combining these two inequalities shows that $\hat f = f^*$ and $\bar f_\rho \to f^*$ in $H$. The latter fact implies that for $\rho$ sufficiently small, $\bar f_\rho \in B_\gamma^H(f^*)$ automatically and hence the feasible set in \eqref{eq:ocProblemPenBarbuA} can be taken to be just $\Uad$. Finally, since the limit point $\hat f = f^*$ is independent of the subsequence that was taken, it follows by the subsequence principle that the entire sequence $\{\bar f_\rho\}$ converges. From this, we also gain convergence for $\{\bar y_\rho\}$ and $\{\bar z_\rho\}$ (by repeating the above arguments).
\end{proof}

\subsection{C-stationarity}\label{sec:CStationarity}
Via \cref{prop:convergenceOfOptimalPoints}, we obtain the existence of minimisers $(y_\rho^*, z_\rho^*, f_\rho^*)$  of \eqref{eq:ocProblemPenBarbu} such that
\[(y_\rho^*, z_\rho^*, f_\rho^*) \to (y^*, z^*, f^*) \text{ in $V \times V \times H$.}\]
Thus, for any $\epsilon > 0$, we can find a $\rho_0$ such that $\rho \leq \rho_0$ implies \[(y_\rho^*, z_\rho^*) \in B_\epsilon(y^*) \times B_\epsilon(z^*).\]
We make the standing assumption \cref{ass:consolidatedAss} \ref{item:PhiDirDiffAndDerivLinearInBall}    on the local differentiability of $\Phi$ and linearity of the derivative on the above balls.
Observe that \eqref{ass:PhiDiffAtU} (which in this context is the assumption that $\Phi$ is differentiable at $\Z_\rho(f_\rho^*)$) follows from these assumptions: since $\Z_\rho(f_\rho^*) \to \Z(f^*)$ (thanks to \cref{thm:jointContinuityOfRhoMapsUnderContraction}), for sufficiently small $\rho$, $\Z_\rho(f_\rho^*) \in B_\epsilon(\Z(f^*))$ and $\Phi$ is differentiable at these points too.

In the next result, we meet the conditions to apply the directional differentiability result of \cref{thm:hadamardDiffMRhoNEW}.
\begin{prop}\label{prop:ssForPenalisedOC}
Let \eqref{ass:PhiCC}
hold.
	For any optimal point $(y_\rho^*, z_\rho^*, f_\rho^*$) of \eqref{eq:ocProblemPenBarbu}, there exists $(p_\rho^*, q_\rho^*) \in V\times V $ such that 
\begin{equation}\label{eq:penalisedOCSystem}
\begin{aligned}
A^*p_\rho^* + \frac1\rho(\Id-\Phi'(y_\rho^*))^*\sigma_\rho'(y_\rho^*-\Phi(y_\rho^*))p_\rho^*  &= -J_y(y_\rho^*, z_\rho^*, f_\rho^*),\\
A^*q_\rho^* + \frac1\rho(\Id-\Phi'(z_\rho^*))^*\sigma_\rho'(z_\rho^*-\Phi(z_\rho^*))q_\rho^*  &= -J_z(y_\rho^*, z_\rho^*, f_\rho^*),\\
\langle  J_f(y_\rho^*, z_\rho^*, f_\rho^*)-p_\rho^* - q_\rho^*, f_\rho^* -v \rangle + (f_\rho^*-f^*, f_\rho^*-v)_H &\leq 0 \quad \forall v \in \Uad.
\end{aligned}
\end{equation}
\end{prop}
\begin{proof}
Defining $\hat J(f) := \bar J(\M_\rho(f),\m_\rho(f), f)$  we consider the reduced problem
\[\min_{f \in \Uad} \hat J(f).\]
Note that we may use the chain rule (e.g., see \cite[Proposition 2.47]{MR1756264}) to differentiate $\hat J$ since it is the composition of a $C^1$ map with a directionally differentiable map. Now, at the optimal point $f_\rho^*$, we have $\hat J(f_\rho^*+sh) - \hat J(f_\rho^*) \geq 0$ for all $h \in \mathcal{R}_{\Uad}(f_\rho^*)$, hence
\[\langle \hat J'(f_\rho^*), h \rangle \geq 0 \quad \forall h \in \mathcal{R}_{\Uad}(f_\rho^*).\]
We calculate, with $y_\rho^* = \M_\rho(f_\rho^*)$ and $z_\rho^* = \m_\rho(f_\rho^*)$,
\begin{align*}
\langle \hat J'(f_\rho^*), h \rangle &= \langle J_y(y_\rho^*, z_\rho^*, f_\rho^*), \M_\rho'(f_\rho^*)(h)\rangle+ \langle J_z(y_\rho^*, z_\rho^*, f_\rho^*), \m_\rho'(f_\rho^*)(h)\rangle  + \langle J_f(y_\rho^*, z_\rho^*, f_\rho^*), h \rangle\\
&= \langle \M_\rho'(f_\rho^*)^*J_y(y_\rho^*, z_\rho^*, f_\rho^*) +  \m_\rho'(f_\rho^*)^*J_z(y_\rho^*, z_\rho^*, f_\rho^*), h\rangle + \langle J_f(y_\rho^*, z_\rho^*, f_\rho^*), h \rangle + (f_\rho^*-f^*, h)_H
\end{align*}
with the adjoint well defined since $\Z_\rho'(f_\rho^*)$ is a bounded linear map thanks to \cref{ass:consolidatedAss} \ref{item:PhiDirDiffAndDerivLinearInBall} (which implies that the derivative satisfies a linear PDE, see \eqref{eq:equationForAlphaRhoNEW}). It is easy to see that the previous inequality in fact holds for all $h \in \mathcal{T}_{\Uad}(f_\rho^*)$ by a simple density argument. 

Defining $\theta_\rho^* := -\left(\M_\rho'(f_\rho^*)^*(J_y(y_\rho^*, z_\rho^*, f_\rho^*)) +  \m_\rho'(f_\rho^*)^*(J_z(y_\rho^*, z_\rho^*, f_\rho^*))\right)$, we write the above as
\[\langle J_f(y_\rho^*, z_\rho^*, f_\rho^*)  - \theta_\rho^*, h \rangle + ( f_\rho^*-f^*, h)_H \geq 0 \quad \forall h \in \mathcal{T}_{\Uad}(f).\]
Take $v \in \Uad$, then $h:= v-f_\rho^*$ is in the tangent cone. With this choice of $h$ we recover 
\[\langle J_f(y_\rho^*, z_\rho^*, f_\rho^*) - \theta_\rho^*, v-f_\rho^* \rangle + (f_\rho^*-f^*, v-f_\rho^*)_H \geq 0 \quad \forall v \in \Uad.\]
Let us characterise each term in $\theta_\rho$. First, observe that 
\[p := \M_\rho'(g)^*(d) \iff A^* p + \frac1\rho(\Id-\Phi'(v_\rho))^*\sigma_\rho'(v_\rho-\Phi(v_\rho))p = d\quad \text{where $v_\rho = \M_\rho(g)$}\]
and a similar formula holds for $\m_\rho'(f)^*(w)$. Note that these adjoint maps (which are solution maps of linear PDEs) are linear in $w$. Hence if we define $p_\rho^*:= \M_\rho'(f_\rho^*)^*(-J_y(y_\rho^*, z_\rho^*, f_\rho^*))$ and $q_\rho^*:= \m_\rho'(f_\rho^*)^*(-J_z(y_\rho^*, z_\rho^*, f_\rho^*))$, they satisfy $\theta_\rho^*=p_\rho^*+q_\rho^*$ and the equations stated in the proposition.
\end{proof}

Before proceeding, let us record some facts. Due to the Lipschitz condition \cref{ass:smallLipschitzForOCProblem} \ref{item:smallLipschitzM}, \ref{item:smallLipschitzm}, we have
\begin{equation}
(\Id-\Phi'(w))\colon V \to V \text{ is invertible for $w \in B_\epsilon(y^*)$ if $J_y \not\equiv 0$, and for $w \in B_\epsilon(z^*)$ if $J_z \not\equiv 0$},\label{ass:newPhiInvInvertible}
\end{equation}
which follows from the Neumann series,
and the inverse satisfies $\norm{(\Id-\Phi'(w))^{-1}v}{V} \leq (1-C_L)^{-1}\norm{v}{V}$ for all $v \in V$.
For an arbitrary $v \in V$, we set $u = (\Id-\Phi'(w))^{-1}v$. Then we have
\begin{align*}
\langle A(\Id-\Phi'(w))^{-1}v, v \rangle &= \langle Au, (\Id-\Phi'(w)) u \rangle \geq C_a'\norm{u}{V}^2
\geq \frac{C_a'}{(1 + C_L)^2} \norm{v}{V}^2
\end{align*}
for some $C_a'$ depending only on $C_L, C_a, C_b$ and the self-adjointedness of $A$, by using \cref{lem:small_estimate} (see also \cite{Wachsmuth2021:2}) adapted to the operator $\Phi'(w)$.
Thus we have shown that
\begin{align}
&A(\Id-\Phi'(w))^{-1}\colon V \to V^* \text{ is uniformly bounded and uniformly coercive}
\label{ass:newAPhiInvCoerciveUniform}
\end{align}
for $w$ belonging to the same sets as in \eqref{ass:newPhiInvInvertible},
see also \cite[Lemmas~3.3, 3.5]{Wachsmuth2019:2}.

\begin{lem}
Under \cref{ass:consolidatedAss} \ref{item:doubleConts}, if $J_y \not\equiv 0$, for sequences $v_n \to v$ and $q_n \weaklyto q$  in $V$ with $v_n, v \in B_\epsilon(y^*)$,  we have 
\begin{equation}
\liminf_{n \to \infty}\langle A(\Id-\Phi'(v_n))^{-1}q_n, q_n \rangle \geq \langle A(\Id-\Phi'(v))^{-1}q, q \rangle \label{ass:weakLowerSC}.
\end{equation}
A similar result holds if $J_z \not\equiv 0$ with the obvious modifications.
\end{lem}
\begin{proof}
Let $T_n := (\Id-\Phi'(v_n))^{-1}$. We have, due to the coercivity above,
\begin{align*}
0 &\leq \langle AT_n(q_n-q), q_n-q \rangle = \langle AT_nq_n, q_n \rangle - \langle AT_n q_n, q \rangle - \langle AT_n q, q_n \rangle + \langle AT_n q, q \rangle
\end{align*}
Rearranging,
\begin{align*}
\langle AT_nq_n, q_n \rangle \geq  \langle AT_n q_n, q \rangle + \langle AT_n q, q_n \rangle - \langle AT_n q, q \rangle
\end{align*}
taking the limit inferior, using on the right-hand side \eqref{ass:doubleContWeak} for the first and last terms and \eqref{ass:doubleCont} for the second term, we obtain the desired statement.
\end{proof}

For convenience and because of structural reasons, the proof of \cref{thm:ocPenalisation} will be realised via the next three propositions. First, we show that a system of so-called \emph{weak C-stationarity} is satisfied, see \cite[\S 5]{AHROCQVI} for the terminology.

\begin{prop}[Weak C-stationarity]\label{prop:weakCStationarity}
There exists $(p^*, q^*,    \lambda^*, \zeta^*) \in   V \times V \times V^* \times V^*$ satisfying 
\begin{subequations}\label{eq:weakStationaritySystemAgain}
\begin{align}
y^* &= \M(f^*),\\
z^* &= \m(f^*),\\
A^*  p^*  + (\Id-\Phi'(y^*))^*\lambda^* &= -J_y(y^*, z^*, f^*),\label{eq:wSS1}\\
A^*q^*  + (\Id-\Phi'(z^*))^*\zeta^* &= -J_z(y^*, z^*, f^*),\label{eq:wSS1b}\\
f^* \in \Uad : \langle J_f(y^*, z^*, f^*)- p^* - q^*, f^*-v\rangle &\leq 0
 \quad \forall v \in \Uad,\label{eq:wSS3}\\
\langle  \lambda^*,  p^* \rangle &\geq 0,\label{eq:wSS5}\\
\langle \zeta^*, q^* \rangle &\geq 0.\label{eq:wSS5b}
\end{align}
\end{subequations}
\end{prop}
In this and the following proofs, for ease of reading, we will omit the stars in $\rho$-dependent notation as $p_\rho^*$ and simply write this as $p_\rho$. 
\begin{proof}
By construction, we already know that $(y_\rho, z_\rho, f_\rho) \to (y^*,z^*,f^*)$ in $V\times V \times H$ due to \cref{prop:convergenceOfOptimalPoints}. We now need to pass to the limit in the system \eqref{eq:penalisedOCSystem} for the adjoint states and the optimal control. We write the arguments just for $p_\rho$; obvious modifications will work for the $q_\rho$ equation too. 

\medskip

\noindent\textit{1. Satisfaction of the equation.} The weak form of the equation for $p_\rho$ is
\begin{align*}
\langle A^*p_\rho, \varphi \rangle + \frac1\rho\langle \sigma_\rho'(y_\rho-\Phi(y_\rho))p_\rho, (\Id-\Phi'(y_\rho))\varphi \rangle  &= -\langle J_y(y_\rho, z_\rho, f_\rho), \varphi \rangle \qquad \forall \varphi \in V.
\end{align*}
By defining $ v:=(\Id-\Phi'(y_\rho))\varphi$,
thanks to the invertibility property \eqref{ass:newPhiInvInvertible}, this can be transformed to 
\begin{align*}
\langle A^*p_\rho, (\Id-\Phi'(y_\rho))^{-1}v \rangle + \frac1\rho\langle \sigma_\rho'(y_\rho-\Phi(y_\rho))p_\rho, v \rangle  &= -\langle J_y(y_\rho, z_\rho, f_\rho), (\Id-\Phi'(y_\rho))^{-1}v \rangle \qquad \forall v \in V.
\end{align*}
Now, selecting $v=p_\rho$, using the coercivity \eqref{ass:newAPhiInvCoerciveUniform}, the monotonicity of $\sigma_\rho$ (which implies that $\langle \sigma_\rho'(v)(h),h \rangle \geq 0$ for all $v, h \in V$), Young's inequality with $\gamma>0$ and the uniform boundedness of $J_y$ (see \cref{ass:smallLipschitzForOCProblem} \ref{item:test}) and of $(\Id-\Phi'(y_\rho))^{-1}$ (see the discussion above \eqref{ass:newAPhiInvCoerciveUniform}), we obtain
\begin{align*}
C_a'\norm{p_\rho}{V}^2 \leq C_\gamma + \gamma\norm{p_\rho}{V}^2.
\end{align*}
Selecting $\gamma$ sufficiently small so that the right-most term is absorbed onto the left, we obtain a 
bound on $\{p_\rho\}$ independent of $\rho$. This gives rise to the convergence (for a subsequence that has been relabelled)
\[p_\rho\weaklyto p.\]
In a similar way, we also obtain $q_\rho \weaklyto q$. Define 
\begin{align*}
 \lambda_\rho &:= \frac 1\rho \sigma_{\rho}'(y_\rho-\Phi(y_\rho))^* p_\rho,\\
 \mu_\rho &:= \frac 1\rho (\Id-\Phi'(y_\rho))^*\sigma_{\rho}'(y_\rho-\Phi(y_\rho))^*p_\rho= -J_y(y_\rho, z_\rho, f_\rho) - A^*p_\rho,
\end{align*}
the latter of which, since the right-hand side converges, satisfies 
\begin{alignat}{3}
 \mu_\rho &\weaklyto  \mu := -J_y(y, z, f)-A^*p
\label{eq:convMuAndXi}.
\end{alignat}
Setting $ \lambda := (\Id-\Phi'(y)^*)^{-1} \mu$ in \eqref{eq:convMuAndXi} we get \eqref{eq:wSS1}.

\medskip

\noindent\textit{2. Inequality relating multiplier to adjoint.}
Again using monotonicity of $\sigma_\rho$, 
\begin{align*}
\langle  J_y(y_\rho, z_\rho, f_\rho)+ A^*p_\rho, (\Id-\Phi'(y_\rho))^{-1}p_\rho\rangle = -\langle  \mu_\rho, (\Id-\Phi'(y_\rho))^{-1}p_\rho\rangle
&= -\frac 1\rho \langle \sigma_{\rho}'(y_\rho-\Phi(y_\rho))^* p_\rho,  p_\rho \rangle \leq 0,
\end{align*}
and taking the limit superior of this,   we obtain (noting that $ (\Id-\Phi'(y_\rho))^{-1}p_\rho \weaklyto  (\Id-\Phi'(y))^{-1}p$ by \eqref{ass:doubleContWeak})
\begin{align*}
0 &\geq \limsup_{\rho \to 0}\langle  J_y(y_\rho, z_\rho, f_\rho) + A^* p_\rho, (\Id-\Phi'(y_\rho))^{-1}p_\rho\rangle \\
&\geq \limsup_{\rho \to 0}\langle  J_y(y_\rho, z_\rho, f_\rho), (\Id-\Phi'(y_\rho))^{-1}p_\rho\rangle  + \liminf_{\rho \to 0} \langle A(\Id-\Phi'(y_\rho))^{-1} p_\rho, p_\rho\rangle \tag{using $\limsup(a_n + b_n) \geq \limsup(a_n) + \liminf(b_n)$}\\
&\geq \langle  J_y(y, z, f), (\Id-\Phi'(y))^{-1}p \rangle +\langle A(\Id-\Phi'(y))^{-1}p, p \rangle \\
&= \langle  -\mu^*, (\Id-\Phi'(y))^{-1}p \rangle
\end{align*}
using the continuity of the Fr\'echet derivative from \cref{ass:smallLipschitzForOCProblem} \ref{item:test} and  \eqref{ass:weakLowerSC} for the final inequality. This shows \eqref{eq:wSS5}.

\medskip

\noindent\textit{3. VI relating control to adjoint.}
 Finally, writing the VI relating $u_\rho$ and $\theta_\rho := p_\rho + q_\rho$ in \eqref{eq:penalisedOCSystem} as
\[0 \leq \langle J_f(y_\rho, z_\rho, f_\rho)- \theta_\rho , v-f_\rho \rangle + (f_\rho - f^*, v-f_\rho)_H = \langle J_f(y_\rho, z_\rho, f_\rho), v-f_\rho \rangle-\langle \theta_\rho, v-f_\rho \rangle 
+ (f_\rho - f^*, v-f_\rho)_H  \quad \forall v \in \Uad\]
 and taking the limit inferior here and using the continuity of $J_f$ from \cref{ass:smallLipschitzForOCProblem} \ref{item:test} and the identity $\liminf_n (a_n + b_n) \leq \limsup_n a_n + \liminf b_n$, we get the desired inequality.
\end{proof}

The next results (till the end of this section) use the fact that $(\cdot)^+\colon V \to V$ is continuous. Furthermore, the next proposition uses weak sequential continuity of the map too.

\begin{prop}[Orthogonality conditions]\label{prop:secondPart}
 We have
\begin{subequations}
\begin{align*}
\langle \xi_1^*, (p^*)^+ \rangle = \langle \xi_1^*, (p^*)^- \rangle = \langle \xi_2^*, (q^*)^+ \rangle = \langle \xi_2^*, (q^*)^- \rangle&= 0.
\end{align*}
\end{subequations} 
\end{prop}
In the proof, we use specific properties of the fact that $H$ is a Lebesgue space. The proof is almost identical to that of \cite[Theorem 5.11]{AHROCQVI} but we give it here for completeness.
\begin{proof}
Let us introduce the sets
\[M_1(\rho) := \{ 0 \leq y_\rho - \Phi(y_\rho) < \epsilon\} \qquad\text{and}\qquad M_2(\rho) := \{y_\rho - \Phi(y_\rho) \geq \epsilon\}.\]
Since $\langle \xi_\rho, y_\rho-\Phi(y_\rho) \rangle \to \langle \xi^*, y-\Phi(y) \rangle = 0$,  we find
\begin{align*}
\frac 1\rho \int_{M_1(\rho)}\frac{(y_\rho - \Phi(y_\rho))^3}{2\epsilon} + \frac 1\rho\int_{M_2(\rho)}\left(y_\rho - \Phi(y_\rho)-\frac \epsilon 2\right)(y_\rho - \Phi(y_\rho)) &\to 0,
\end{align*}
and as both integrands  are non-negative, 
\begin{equation}\label{eq:pr1}
\norm{\frac{\chi_{M_1(\rho)}(y_\rho - \Phi(y_\rho))^{\frac 32}}{\sqrt{\rho \epsilon}}}{} \to 0\qquad\text{and}\qquad \norm{\frac{\chi_{M_2(\rho)}(y_\rho - \Phi(y_\rho)-\frac\epsilon 2)}{\sqrt{\rho}}}{} \to 0,
\end{equation}
where for the second convergence we used the fact that $y_\rho - \Phi(y_\rho) \geq y_\rho - \Phi(y_\rho)-\epsilon\slash 2 \geq 0$.
We calculate
\begin{align}
\nonumber \langle \xi_\rho, p_\rho \rangle &= \frac 1\rho \int_{M_1(\rho)}\frac{(y_\rho-\Phi(y_\rho))^2}{2\epsilon}p_\rho  + \frac 1\rho \int_{M_2(\rho)}\left(y_\rho-\Phi(y_\rho)-\frac \epsilon 2\right)p_\rho\\
&\leq \frac 12\norm{\chi_{M_1(\rho)}\frac{(y_\rho-\Phi(y_\rho))^{3\slash 2}}{\sqrt{\rho\epsilon}}}{}\norm{\frac{(y_\rho-\Phi(y_\rho))^{1\slash 2}}{\sqrt{\rho\epsilon}}\chi_{M_1(\rho)}p_\rho}{}  + \norm{\frac{\chi_{M_2(\rho)}\left(y_\rho-\Phi(y_\rho)-\frac \epsilon 2\right)}{\sqrt{\rho}}}{}\norm{\frac{\chi_{M_2(\rho)}p_\rho}{\sqrt{\rho}}}{}.\label{eq:productXiP}
\end{align}
Now, using \eqref{eq:pr1}, the first factor in each term above converges to zero and hence the above right-hand side will converge to zero if we are able to show that the second factor in each term remains bounded.
Since $\mu_\rho$ and $(\Id-\Phi'(y_\rho))^{-1}p_\rho$ are bounded (for the latter, see \eqref{ass:newPhiInvInvertible} and the discussion), so is their duality product, and therefore
\begin{align*}
\nonumber C &\geq |\langle \mu_\rho, (\Id-\Phi'(y_\rho))^{-1}p_\rho \rangle|\\
\nonumber  &=\frac 1\rho \left|\int_\Omega  \sigma_\rho'(y_\rho-\Phi(y_\rho))(p_\rho)^2\right|\\
&=\frac{1}{\rho}\int_{\Omega}\chi_{M_1(\rho)}\frac{y_\rho-\Phi(y_\rho)}{\epsilon}(p_\rho)^2 + \frac{1}{\rho}\int_{\Omega}\chi_{M_2(\rho)}(p_\rho)^2.
\end{align*}
Both of the terms on the right-hand side are individually bounded uniformly in $\rho$ as the integrands are non-negative. This fact then implies from \eqref{eq:productXiP} that 
\[\langle \xi^*, p^* \rangle = 0.\]
Replacing $p_\rho$ by $(p_\rho)^+$ in \eqref{eq:productXiP} and in the above calculation, we also obtain in the same way (utilising the fact\footnote{This is due to the compact embedding $V \ctsCompact H$ and the fact that $(\cdot)^+\colon H \to H$ is continuous as well as the boundedness of $(\cdot)^+\colon V \to V$ that we assumed in the introduction.} that $v_n \weaklyto v$ in $V$ implies that $v_n^+ \weaklyto v^+$ in $V$)
\[\langle \xi^*, (p^*)^+ \rangle = 0.\]
\end{proof}

We are left to show the conditions \eqref{eq:eaGerd1} and \eqref{eq:eaGerd2} on the multipliers.  To do so, we will follow an approach motivated by \cite[Lemma~2.6]{Wachsmuth2014:2}. 
\begin{lem}
	\label{lem:limit_process_multiplier}
	If $g_n \weaklyto g$ in $V^*$
	and $s_n \to s$ in $V$ with
	$s_n \ge 0$ and
	\begin{equation*}
		\dual{g_n}{v} = 0
		\qquad\forall v \in V, \;\;0 \le v \le s_n,
	\end{equation*}
then
	\begin{equation*}
		\dual{g}{v} = 0
		\qquad\forall v \in V, \;\; 0 \le v \le s.
	\end{equation*}
\end{lem}
\begin{proof}
	Let $v \in V$ with $0 \le v \le s$ be given. Set $v_n:= \inf(v,s_n)$, which satisfies $0 \leq v_n \leq s_n$ and $v_n \to v$.
	Thus,
	\begin{equation*}
		0
		=
		\dual{g_n}{\inf(v,s_n)}
		\to
		\dual{g}{v}.
	\end{equation*}
\end{proof}
In the next lemma, we use the fact that 	\begin{equation*}
		\sigma_\rho( z ) = \sigma_\rho( z - v )
		\quad
		\forall v \in V, \;\; 0 \le v \le z^-.
	\end{equation*}
This essentially means  that $\sigma_\rho(z)$ ignores changes of $z$ in the regions where $z$ is already negative.

\begin{lem}
	\label{lem:complementarity_condition}
	We have
	\begin{equation*}
		\dual{\lambda^*}{v} = 0
		\quad
		\forall v \in V, \;\;0 \le v \le \Phi(y^*)- y^*
		.
	\end{equation*}
\end{lem}

The condition on $\lambda^*$ means,
roughly speaking,
that $\lambda^*$ vanishes
on the inactive set on which $y^* - \Phi(y^*) < 0$.

\begin{proof}
	The property on $\sigma_\rho$ stated above immediately implies
	\begin{equation*}
		\sigma_\rho'(z) v = 0
		\quad
		\forall v \in V, \;\; 0 \le v \le z^-.
	\end{equation*}
	Using the definition of $\lambda_\rho$, we find
	\begin{equation*}
		\dual{\lambda_\rho}{v}
		=
		\frac1\rho \dual{\sigma_{\rho}'(y_\rho-\Phi(y_\rho)) v }{p_\rho}
		=
		0
		\quad\forall v \in V, \;\;  0 \le v \le (y_\rho-\Phi(y_\rho))^-
	\end{equation*}
	by the above property of $\sigma_\rho'$.
	As $(y_\rho-\Phi(y_\rho))^- \to (y^* -\Phi(y^*))^- = \Phi(y^*) - y^*$,
	\cref{lem:limit_process_multiplier} yields the claim.
\end{proof}
Obviously, a similar condition also holds for $\zeta^*$.

\begin{prop}\label{prop:multiplerConditionsGerd}
Let \cref{ass:forCStationarity} hold. We have
\begin{align*}
 \langle \lambda^*, v \rangle &= 0  \text{  $\forall v \in V : v = 0$ q.e. on $\{y^* = \Phi(y^*)\}$},\\
 \langle \zeta^*, v \rangle &= 0  \text{  $\forall v \in V : v = 0$ q.e. on $\{z^* = \Phi(z^*)\}$}.
\end{align*}
Hence, the  system \eqref{eq:cStationaritySystem} is satisfied.
\end{prop}
\begin{proof}
Set $\hat y := \Phi(y^*) - y^*$. Since $\hat y \in V$, it has a quasi-continuous representative and we will identify
$\hat y$ with its representative. Define the active set
\[A := \{ \hat y = 0 \}.\] 
Let $v \in V$ with $v \geq 0$ and $v = 0$ q.e. on $A$ be given. Since we have the following expression for the tangent cone of $V_+$ (see \cite[Theorem 6.57]{MR1756264} in the $V=H^1_0(\Omega)$ setting or \cite[Lemme 3.2]{MR0423155} in the general Dirichlet space setting):
\[\mathcal{T}_{V_+}(\hat y) = \{ \varphi \in V : \varphi \geq 0 \text{ q.e. on $A$} \},\]
it follows that $v \in \mathcal{T}_{V^+}(\hat y)$ and hence, there exists a sequence $\{v_n\}$ with $v_n \to v$ in $V$ and $v_n \le t_n \hat y$ for some $t_n > 0$. Thus,
    \[0 \le \max( 0, v_n / t_n ) \le \hat y\]
and we can apply the conclusion of \cref{lem:complementarity_condition} and get
    \[ \langle \lambda^*, \max(0, v_n / t_n) \rangle = 0.\]
Multiplying by $t_n$ and passing to the limit $n \to \infty$ gives
  \[  \langle \lambda^*, v \rangle = 0   \text{ for all $v \ge 0$, $v = 0$ q.e. on $A$.}
\]
Then, using the decomposition $v=v^+ - v^-$, we obtain the result.
\end{proof}

\begin{remark}[$\mathcal{E}$-almost C-stationarity]\label{prop:epsAlmostCStationarity}
Define the inactive sets
\[\mathcal{I}_1 = \{ y^* < \Phi(y^*) \} \quad\text{and}\quad\mathcal{I}_2 = \{ z^* < \Phi(z^*) \}.\]
The following argument shows that the conditions
\begin{align}
\forall \tau > 0, \exists E^\tau \subset \mathcal{I}_1 \text{ with } |\mathcal{I}_1 \setminus E^\tau| \leq \tau : \langle \lambda^*, v \rangle &= 0 \quad \forall v \in V  : v = 0 \text{ a.e. on $\Omega \setminus E^\tau$}\label{eq:eaEaaa},\\
\forall \tau > 0, \exists E^\tau \subset \mathcal{I}_2 \text{ with } |\mathcal{I}_2 \setminus E^\tau| \leq \tau : \langle \zeta^*, v \rangle &= 0 \quad \forall v \in V  : v = 0 \text{ a.e. on $\Omega \setminus E^\tau$}\label{eq:eaEBaa},
\end{align}
are an easy consequence of \cref{prop:multiplerConditionsGerd} and of the regularity of the Lebesgue measure.
For every $\tau > 0$, there exists an open set $O^\tau$
with
$\set{y^* = \Phi(y^*)} \subset O^\tau$
and
$\abs{ O^\tau \setminus \set{y^* = \Phi(y^*)} } \le \tau$.
Using $\mathcal I_1 = \Omega \setminus \set{y^* = \Phi(y^*)}$
and $E^\tau := \Omega \setminus O^\tau$
we get
$E^\tau \subset \mathcal I_1$
and
$\abs{ \mathcal I_1 \setminus E^\tau } \le \tau$ by taking complements.
Next, we take an arbitrary function $v \in V$ with $v = 0$ a.e.\ on $\Omega \setminus E^\tau = O^\tau$.
Since $O^\tau$ is open, this gives $v = 0$ q.e.\ on $O^\tau$
and, in particular, $v = 0$ q.e.\ on $\set{y^* = \Phi(y^*)}$.
Thus, \cref{prop:multiplerConditionsGerd} yields
$\dual{\lambda^*}{v} = 0$ and
we get \eqref{eq:eaEaaa}.
\end{remark}

\begin{remark}[Regularity of optimal control]Suppose we have $J_f(y,z,f) = \nu f$ and we take $\Uad$ to be of the box constraint type
\[\Uad = \{ u \in H: u_a \leq u \leq u_b \text{ a.e. in $\Omega$}\}\]
for given functions $u_a, u_b \in H$.  The VI relating $f^*$ and $p^*$ is, in this case,
\[f^* \in \Uad : \langle \nu f^* - p^*-q^* , f^*-v\rangle \leq 0
 \quad \forall v \in \Uad,\]
Using the characterisation in \cite[\S II.3]{Kinderlehrer},
\[\frac 1\nu (p^*+q^*)+ \left(u_a-\frac{p^*+q^*}{\nu}\right)^+-\left(\frac{p^*+q^*}{\nu}-u_b\right)^+ = f^*\]
and it follows that $f^* \in V$ if $u_a$ and $u_b$ belong to $V$.
\end{remark}

\subsection{Alternative stationarity conditions}\label{sec:penalisationOfQVI}

In some papers, e.g. \cite{MR3056408}, in direct analogy with the finite dimensional setting, rather than the inequality condition \eqref{eq:eaCs}, the stronger condition
\begin{equation*}
\langle \lambda^*, \psi p^* \rangle \geq 0 \quad \text{for all sufficiently smooth and non-negative $\psi$}
\end{equation*}
is required in order to satisfy the terminology \textit{C-stationarity}. We can show this holds under an additional assumption.
\begin{prop}[Satisfaction of alternative criterion in C-stationarity]\label{prop:altCondition}Under the conditions of
\cref{thm:ocPenalisation}, assume also that for $q_\rho \weaklyto q$ in $V$,
\begin{equation}\label{ass:wlscForPhiInvForSCS}
\liminf_{\rho \to 0} \langle A^*q_\rho, (\Id-\Phi'(y_\rho^*))^{-1}(\psi q_\rho) \rangle  \geq \langle A^*q, (\Id-\Phi'(y^*))^{-1}(\psi q) \rangle\quad \forall \psi \in W^{1,\infty}(\Omega) \text{ with } \psi \geq 0.
\end{equation}
Then the inequality condition \eqref{eq:eaCs} can be strengthened to
\[\langle \lambda^*, \psi p^*\rangle \geq 0 \quad \forall \psi \in W^{1,\infty}(\Omega) \text{ with } \psi \geq 0.\]
Under the obvious modifications to the above assumption, \eqref{eq:eaCsB} can also be strengthened similarly.
\end{prop}
\begin{proof}
Testing the equation for $p_\rho$ with $(\Id-\Phi'(y_\rho))^{-1}(\psi p_\rho) $, noticing that $\psi p_\rho \weaklyto \psi p$ in $V$ and arguing in a similar way to the proof of \cref{prop:weakCStationarity},
\begin{align*}
\limsup_{\rho \to 0}-\langle \mu_\rho, (\Id-\Phi'(y_\rho))^{-1}(\psi p_\rho) \rangle  
&= \limsup_{\rho \to 0}\langle J_y(y_\rho,z_\rho,f_\rho), (\Id-\Phi'(y_\rho))^{-1}(\psi p_\rho)\rangle\\
&\quad   + \liminf_{\rho \to 0}\langle A^*p_\rho, (\Id-\Phi'(y_\rho))^{-1}(\psi p_\rho) \rangle\\
&\geq \langle J_y(y,z,f), (\Id-\Phi'(y))^{-1}(\psi p) \rangle  + \langle A^*p, (\Id-\Phi'(y))^{-1}(\psi p) \rangle\tag{using \cref{ass:smallLipschitzForOCProblem} \ref{item:test} and \eqref{ass:wlscForPhiInvForSCS}}\\
&= -\langle \mu, (\Id-\Phi'(y))^{-1}(\psi p)\rangle\\
&= -\langle \lambda, \psi p\rangle.
\end{align*}
On the other hand, we have
\[ \langle \mu_\rho, (\Id-\Phi'(y_\rho))^{-1}(\psi p_\rho) \rangle = \langle \lambda_\rho, \psi p_\rho \rangle  = \frac 1\rho \int_\Omega \sigma_\rho'(y_\rho-\Phi(y_\rho))(p_\rho)^2\psi\geq 0\]
which implies the result. 
\end{proof}

\begin{remark}
Some works (such as \cite{MR2822818}) call the system \eqref{eq:cStationaritySystem} C-stationarity only if the `q.e.' in conditions \eqref{eq:eaGerd1} and \eqref{eq:eaGerd2} are replaced by `a.e'. Note that this is a stronger condition.
\end{remark}

\section{Conclusion}\label{sec:conclusion}
In conclusion, we have provided a thorough theory of Lipschitz and differential stability for $\M$ and $\m$ and the penalised versions. We studied in depth the penalised problem \eqref{eq:penalisedPDEGeneralMRho} and its properties and used it to derive stationarity conditions for a general class of optimisation problems with the extremal maps as constraints. We conclude with some remarks.

\begin{itemize}
\item Applying this theory to other real-world phenomena (such as  applications in biomedicine \cite{2021arXiv211002817S}) in this context and studying numerical schemes in line with \cref{remark:constructiveApproxOfExtremals} are natural next steps.

\item It would be interesting to derive strong stationarity conditions for \eqref{eq:ocProblemMostGeneral} using the approaches of \cite{Wachsmuth2013:2, AHROCQVI}.

\item Resolving whether in \cref{thm:convergenceOfMRho} $\m_\rho(f)$ indeed converges (weakly) to $\m(f)$ or providing a counterexample is open.

\item We aim to investigate whether the convergence result for $\M_\rho(f)$  in \cref{thm:convergenceOfMRho} can be used to obtain differentiability results for $\M$ without the small Lipschitz assumption \eqref{ass:smallLipschitzForOCProblem}.

\end{itemize}

\section*{Acknowledgements}
We thank the referees for their helpful comments. CNR was partially supported by NSF grant DMS-2012391.

\printbibliography

\end{document}